\documentclass[a4paper]{amsart}

\usepackage{amsthm,amsfonts,amsmath,amssymb,mathtools,xcolor}
\usepackage{hyperref}

\newtheorem{theorem}{Theorem}[section]
\newtheorem{lemma}[theorem]{Lemma}
\newtheorem{proposition}[theorem]{Proposition}
\newtheorem{corollary}[theorem]{Corollary}

\theoremstyle{definition}
\newtheorem{definition}[theorem]{Definition}
\newtheorem{remark}[theorem]{Remark}

\numberwithin{equation}{section}

\newcommand{\eps}{{\varepsilon}}
\newcommand{\ra}{\right\rangle}
\newcommand{\la}{\left\langle}
\newcommand{\de}{\partial}
\newcommand{\ph}{\varphi}
\newcommand{\Lip}{{\rm {Lip}}}

\newcommand{\Id}{{\rm {Id}}}
\newcommand{\dist}{{\rm {dist}}}
\newcommand{\diam}{{\rm {diam}}}
\newcommand{\cA}{{\mathcal{A}}}

\newcommand{\cI}{{\mathcal{I}}}

\newcommand{\cF}{{\mathcal{F}}}
\newcommand{\cG}{{\mathcal{G}}}

\newcommand{\cH}{{\mathcal{H}}}
\newcommand{\cL}{{\mathcal{L}}}
\newcommand{\cN}{{\mathcal{N}}}

\newcommand{\cS}{{\mathcal{S}}}

\newcommand{\cX}{{\textup{Lip}_{1}}}
\newcommand\R{{\mathbb R}}
\newcommand\N{{\mathbb N}}
\newcommand\Z{{\mathbb Z}}
\newcommand\Q{{\mathbb Q}}
\DeclareMathOperator*{\esssup}{ess\,sup}

\newcommand\sS{\mathbb{S}}

\title{Equidimensional isometric maps}
\author[B. Kirchheim, E. Spadaro, L. Sz\'ekelyhidi Jr.]{Bernd Kirchheim, Emanuele Spadaro, L\'aszl\'o Sz\'ekelyhidi Jr.}
\date{}

\begin{document}
\begin{abstract}
In Gromov's treatise (\textit{Partial differential relations}, volume 9 of \textit{Ergebnisse der Mathematik und ihrer Grenzgebiete (3)}, 1986),
a continuous map between Riemannian manifolds is called isometric if it preserves the length of rectifiable curves.
In this note we develop a method using the Baire category theorem for constructing such isometries. We show that a typical $1$-Lipschitz map is isometric in canonically formulated extension and restriction problems.
\end{abstract}

\maketitle

\tableofcontents

\section{Introduction}
Since the fundamental works of Nash \cite{Nash} and Kuiper \cite{Kuiper} it is well known that isometric maps with low regularity can be surprisingly flexible objects. In particular, any short immersion of an $n$-dimensional Riemannian manifold with continuous metric into $\R^{n+1}$ can be uniformly approximated by isometric immersions of class $C^1$.
One of the main ideas introduced by Nash, and revisited by Kuiper, is an iterative scheme, whereby in each stage the short map is perturbed by a rapidly oscillating ``corrugation'' (or ``spiral'' in higher codimensions) such that the resulting maps converge in $C^1$ to an isometric immersion.

On the contrary, in the equidimensional case, that is, for maps from a $n$-dimensional manifold into $\R^n$, isometries of class $C^1$ are \textit{rigid}.
Namely, if $f:\R^n\to\R^n$ is a $C^1$ map with $Df \in O(n)$ for every $x \in \R^n$, then $f$ is globally orientation preserving or reversing and, by a classical Liouville theorem, is an affine map, i.e.~a rigid motion.

Therefore, in order to see some flexibility, one needs to relax the $C^1$ condition. A natural choice is to consider 
Lipschitz maps instead. To fix ideas consider maps
$f:\R^n\to\R^n$.
There are several ways in which one can define what it means to be an isometry: either look at changes in the metric under $f$ (a local condition), or look at the effect on the length of curves (a global condition). For $f\in C^1$ 
the two conditions lead to the same notion - this can be seen as a simple example of the local-to-global principle in geometry. If $f$ is merely Lipschitz, by Rademacher's theorem the derivative $Df(x)$ exists for almost every $x\in\R^n$, hence a weak preservation of the metric amounts to the condition 
 \begin{equation}\label{e:weakisometry}
(Df)^T \,Df=\Id\quad\cL^n\textrm{- a.e.~in } \R^n.
\end{equation}
Here we denote by $\cL^n$ the Lebesgue measure on $\R^n$. We will call such maps \emph{weak isometries}. As pointed out by Gromov on p.~218 of his treatise \cite{GromovBook}, such maps might collapse whole submanifolds to a single point and thus are very far from a truly geometric notion of isometry. For instance, it is possible to solve the Dirichlet problem $Df^T \,Df=\Id$ a.e.~in $\Omega=[0,1]^n$ and $f\vert_{\de \Omega} =0$ -- see e.g.~\cite{Cellina:1995uv,DM97}. By extending $f$ periodically on the whole $\R^n$, one can then find a solutions to \eqref{e:weakisometry} such that $f(\R^{n-1}\times\{0\})=\{0\}$.

The more geometric definition of isometry therefore is the following: a Lipschitz map between Riemannian manifolds $f:M\to N$ is \emph{isometric} if it preserves the length of any rectifiable curve (c.f. \cite[\S 2.4.10]{GromovBook}):
\begin{equation}\label{e:isometry}
\ell_M(\gamma) = \ell_N(f\circ\,\gamma) \quad \text{for every }\;\gamma:[0,1] \to M\; \text{rectifiable}.
\end{equation}
It is not difficult to see that any isometry is a weak isometry, but the converse is in general false. To compare with \eqref{e:weakisometry}, notice that an isometric map $f:\R^n\to \R^n$ satisfies 
\begin{equation}\label{e:diff-inclusion}
(D^M f)^T \, D^Mf = \Id \quad \cH^m\text{-a.e.~on } M,
\end{equation}
for every $m$-dimensional submanifold $M\subset\R^n$, $m=1,\ldots, n$, where $D^M$ denotes the tangential derivative and $\cH^m$ is the $m$-dimensional Hausdorff measure. Actually, it is not difficult to see that in condition \eqref{e:diff-inclusion} it suffices to check the lowest dimensional case $m=1$, i.e.
\begin{equation}\label{e:diff-inclusion1}
|\nabla_\tau f| = 1 \quad \cH^1\text{-a.e.~on }\gamma
\end{equation}
for every rectifiable curve $\gamma\subset \R^n$, where $\nabla_\tau f$ denotes the tangential derivative.

\medskip

For constructing isometries one might imagine a ``folding up'' pattern as the analogous perturbations to corrugations in an iterative scheme \textit{\`a la} Nash and prove results similar in spirit to the Nash-Kuiper theorem.
Indeed, in \cite{GromovBook} Gromov shows that every strictly short map between Riemannian manifolds admits an arbitrarily close uniform approximation by isometries. 
More generally, Gromov's convex integration is a powerful generalization of the Nash technique, that applies to a large class of differential relations. A version for differential inclusions of Lipschitz maps has been developed in 
\cite{MS96,MS99}, where also the system \eqref{e:weakisometry} is treated as a particular case. 

On the other hand it was noticed by several authors \cite{Cellina:2005tx, DM97,Kirchheim}, that the Baire category method, introduced in \cite{DeBlasi:1982tk,BressanFlores} for ordinary differential inclusions, can be applied to problems 
such as \eqref{e:weakisometry} (which can be written as the differential inclusion $Du(x)\in O(n)$ a.e.~$x$). 
This approach leads not only to the density of weak isometries but also to genericity in the sense of Baire category. 

\medskip

Our contribution in this paper is twofold. First of all we develop a version of the Baire category method for isometric maps satisfying \eqref{e:isometry} in the sense considered by Gromov and prove several residuality results. Our method allows one to reduce the problem of Baire-residuality to the density of certain approximate isometries, see \S~\ref{s.approx} below. 

Secondly, we give a self-contained proof of the density of (approximate) isometries that follows the general philosophy of Baire category techniques for differential inclusions. To explain this, recall that the density of Lipschitz isometries between Riemannian manifolds follows from Gromov's result \cite[\S 2.4.11]{GromovBook} concerning the fine approximability of isometries. Alternatively, in $\R^n$ one can use the following result of Brehm \cite{Brehm:1981dg} concerning the extension of isometries:

\begin{theorem}[Brehm \cite{Brehm:1981dg}]\label{t.brehm}
Let $H\subset\R^n$ be a finite set and $f:H\to \R^m$ be a short map, with $n\leq m$. Then, there exists an extension of $f$ to a piecewise affine isometric map of the whole $\R^n$.
\end{theorem}

Both Gromov's and Brehm's proof rely on the (global) geometric property of being an isometry, in particular special piecewise affine isometries (called normally folded maps in \cite{GromovBook}) are used as the basic building block and it is not clear how to generalize this notion to other differential inclusions. In contrast, our approach is to treat isometries as solutions to a \emph{fine differential inclusion} as in \eqref{e:diff-inclusion}, where the tangential derivative on lower-dimensional objects is prescribed. As in the usual Baire category method, we use an explicit oscillating perturbation to show the perturbation property for the (tangential) gradient of the map $f$. The new key point however is to use a calibration to control the underlying curves. We expect our method to find applicability in a more general class of such fine differential inclusions.

To conclude this introduction we mention that there is yet another, stronger notion of isometry. In \cite[\S 2.4.10]{GromovBook} a map $f:M\to N$ between Riemannian manifolds is called a {\it strong isometry} if for any $x,y\in M$
$$
\dist_M(x,y)=\lim_{\eps\to 0}\inf\Bigl\{\sum_{i=0}^{k-1}\dist_N(f(x_i),f(x_{i+1}))\Bigr\},
$$
where the infimum is taken over all $\eps$-chains between $x$ and $y$, that is, sequence of points $x_0=x,\,x_1,\dots,x_k=y$ with $\dist_M(x_i,x_{i+1})\leq \eps$. The same notion is called an {\it intrinsic isometry} in \cite{Petrunin}.
It is not difficult to see that a strong isometry is an isometry. Moreover, strong isometries preserve the length of any curve (not just rectifiable). Now, using Gromov's theorem (or our Theorem \ref{t.extension} below) it is possible to construct an isometry $f:\R^2\to\R^2$, which maps the Koch curve (or any purely unrectifiable curve) to a single point. Such a map will obviously not be a strong isometry. We note in passing that in \cite[\S 2.4.10]{GromovBook} this construction is described with a curve $C$ with the property that $\dim_H(C\cap C_0)<1$ for all rectifiable curves $C_0$. This property is stronger than being purely unrectifiable, and in fact it turns out that such a curve $C$ does not exist - see \cite{BalkaHarangi}.
Our main results and techniques in this paper, in particular in \S \ref{s:density}, do not extend to strong isometries.

\section*{Acknowledgments}
We would like to express our thanks to Giovanni Alberti, to whom we are indebted for many
fruitful and inspiring discussions concerning this work.

\section{Statement of the main results}

We first consider the problem of extending a map defined on an arbitrary compact set $K\subset\R^n$. This is a generalization of the Dirichlet problem on a bounded domain $\Omega\subset\R^n$, if we take $K = \de\Omega$.

It is clear that an isometric extension need not always exist.
For example, consider the following map: $K = \de [0,1]^2 \subset\R^2$ and $f:K\to\R^2$ given by $f(x,y) = (x,0)$.
Clearly, $f$ is a short map admitting a unique $1$-Lipschitz extension to $[0,1]^2
$ (namely $f(x,y) = (x,0)$), which is not an isometric map because, for instance, vertical line segments are mapped to single points.

In order to deal with this issue, we need to characterize the set $C(f,K)$ where the map $f$ has a unique $1$-Lipschitz extension. It is clear that $f$ extends uniquely as a $1$-Lipschitz map on the set
\[
C(f,K) := \bigcup_{H\in \cS} \textup{conv}(H),
\]
where $\cS := \{ H\subset K : f\vert_H \;\textup{is an affine isometry}\}$. As seen in the example above, if the unique $1$-Lipschitz extension on $C(f,K)$ is not isometric, there is no chance to solve the extension problem. On the other hand, if $C(f,K)=K$, the map $f$ does admit extensions which are locally strictly short outside $K$. This is the content of the 
following proposition:
\begin{proposition}\label{p.LSSE0}
A function $f:K\to\R^n$ admits an $1$-Lipschitz extension $h:\R^n\to\R^n$ such that
\begin{itemize}
\item $h\vert_K=f$;
\item $\Lip(h\vert_A) <1$ for every $A \subset\subset \R^n \setminus K$
\end{itemize}
if and only if 
\begin{equation}\label{e:C=K}
C(f,K)=K.
\end{equation} 
Moreover, \eqref{e:C=K} is a generic property in the sense that a typical $1$-Lipschitz map $f:K\to\R^n$ satisfies it.
\end{proposition}
The proof of Proposition \ref{p.LSSE0} (restated as Proposition \ref{p.lsse condition} and \ref{p.LSSE}) is contained in Sections \ref{s:LSSE} and \ref{s:restriction}.
As a consequence, we prove that the solutions to the Dirichlet problem which are isometric in $\R^n\setminus C(f,K)$ are in fact residual:

\begin{theorem}[Typical extension]\label{t.extension}
Let $K\subset\R^n$ be a compact set and $f:K\to\R^n$ a short map. Then, the typical $1$-Lipschitz extension of $f$ to the whole $\R^n$ is isometric on $\R^n\setminus C(f,K)$.
\end{theorem}

We then consider the problem of Dirichlet data $f:K \to \R^n$ which extend to a \textit{global} isometric map $F: \R^n \to \R^n$ (not just of $\R^n \setminus C(f, K)$).
We prove that also this is a generic property.

\begin{theorem}[Typical restriction]\label{t.restriction}
Let $K\subset\R^n$ be a compact set. The typical short map $f:K\to\R^n$ is the restriction of an isometric map of the whole $\R^n$.
\end{theorem}

Finally, we address the problem of isometric maps from a Riemannian manifold $M^n$ into $\R^n$.
We show that such maps are residual in the space of short maps.

\begin{theorem}[Typical isometries]\label{t.immersion}
Let $M$ be a $n$-dimensional Riemannian manifold with continuous metric. Then, the isometric maps of $M$ into $\R^n$ are residual in the space of short maps.
\end{theorem}

\section{Approximate isometric maps}\label{s.approx}

In what follows $M$ is a connected $n$-dimensional smooth manifold
with or without boundary.
We assume that $M$ is endowed with a continuous Riemannian metric $g$; we
denote by $d_M$ and $|\cdot|_g$ the induced Riemannian
distance on $M$ and the norm on each tangent space $T_xM$, respectively.
In the case of subsets of $\R^n$, we use the usual notation $|x|$ and
$x\cdot y$ for the norm and the scalar product of vectors,
respectively.

Given a path-connected subset $S\subseteq M$ we introduce
the following notation.
\begin{itemize}
\item[(a)] The space of short maps from $M$ into $\R^n$ is denoted by $\cX(M,\R^n)$, i.~e.
\[
\cX(M,\R^n):=\big\{f:M\to \R^n\,:\,\Lip_g(f)\leq 1\big\},
\]
where
\[
\Lip_g(f):=\sup_{x\neq y\in M}\frac{|f(x)-f(y)|}{d_M(x,y)}.
\]
\item[(b)] $\Gamma_S(x,y)$ is the set of rectifiable curves from $x$
to $y$ contained in $S$:
\[
\Gamma_S(x,y):=\big\{\gamma:[0,1]\to
S\,:\,\gamma\;\text{rectifiable},\;\gamma(0)=x,\gamma(1)=y\big\}.
\]
We denote by $d_S$ the induced metric, i.e.
\[
d_S(x,y) := \inf_{\gamma \in \Gamma_S(x,y)} \ell_g(\gamma).
\]
\item[(c)] We denote by $\cI(S)$ the set of all short maps $f\in \cX(M,\R^n)$
which are isometric in $S$, i.~e.~$\ell(f\circ\gamma)=\ell_g(\gamma)$ for every
rectifiable curve $\gamma:[0,1]\to S$,
where
\[
\ell(f\circ\gamma)=\int_0^1|(f\circ\gamma)^\prime(t)|\,dt
\quad\text{and}\quad
\ell_g(\gamma)=\int_0^1|\gamma^\prime(t)|_{g(\gamma(t))}\,dt.
\]
Equivalently, $f\in \cI(S)$ if for every $\gamma$ as above
\[
|(f\circ\gamma)^\prime(t)|=|\gamma^\prime(t)|_{g(\gamma(t))}\quad\text{for a.e.
} t\in[0,1].
\]
\item[(d)] For every $\eps>0$ and $x,y\in S$, we
denote by $F_{\eps}(x,y,S)\subset \cX(M,\R^n)$ the mappings satisfying
\begin{align*}
F_{\eps}(x,y,S):=\Big\{f\in \cX(M,\R^n):\,\ell&(f\circ \gamma)+\eps\,\ell_g(\gamma)>
(1-\eps)\,d_S(x,y)\\
& \quad\forall\gamma\in\Gamma_S(x,y)\Big\}.
\end{align*}
\end{itemize}
Note that in general the maps in $\cX(M,\R^n)$ are not bounded (except when $M$ itself is
bounded). 
For this reason, we use the following metric on $\cX(M,\R^n)$:
\[
D(f,g):=\sup_{x\in M}\min\big\{1,|f(x)-g(x)|\big\}=
\min\left\{1,\sup_{x\in M}|f(x)-g(x)|\right\}.
\]
It is easy to verify that $(\cX(M,\R^n),D)$ is a complete metric space and that $D$ induces the
uniform convergence, i.e.
\[
\lim_{l\to+\infty}D(f_l,f)=0
\quad\Longleftrightarrow\quad
\lim_{l\to+\infty}\|f_l-f\|_{C^0(M)}=0.
\]

\bigskip

\begin{definition}
Let $S\subset M$ be path-connected. We define the set of $\eps$-\textit{approximate} isometric maps in $S$ by:
\begin{align}\label{e.app iso}
\cI_\eps(S) &:= \bigcap_{x\neq y \in S} F_\eps(x,y,S).
\end{align}
\end{definition}

The name is justified by the following result.

\begin{lemma}\label{l.iso}
Let $S \subset M$ be path-connected. Then
\begin{equation}\label{e.approx-iso}
\bigcap_{\eps>0}\cI_{\eps}(S) = \cI(S).
\end{equation}
\end{lemma}

\begin{proof}
Note first that $\cI(S)\subset F_\eps(x,y,S)$
for every $\eps>0$ and $x\neq y\in S$.
Indeed, every $f\in\cI(S)$ satisfies
\[
\ell(f\circ\gamma)+\eps\,\ell_g(\gamma)=(1+\eps)\,\ell_g(\gamma)>(1-\eps)\,d_M(x
,y) \quad\forall\,\gamma\in\Gamma_S(x,y).
\]

In order to prove the converse inclusion, assume $f\in\cI_{\eps}(S)$ for every $\eps>0$ and let $\gamma:[0,1]\to S$ be a rectifiable curve. Then, for every partition $0=t_0<\ldots < t_m=1$, setting $\gamma_j := \gamma\vert_{[t_j,t_{j+1}]}$, we have 
\begin{align*}
\ell(f\circ\gamma) = \sum_{j=0}^{m-1}\ell(f \circ \,\gamma_j)\geq \sum_{j=0}^{m-1}d_M(\gamma(t_j),\gamma(t_{j+1})).
\end{align*}
Since this holds for any partition, $\ell(f\circ\gamma) \geq \ell_g(\gamma)$ and, hence, $f\in \cI(S)$.
\end{proof}

\subsection{Separability}
We show next that it suffices to take a countable intersection in order to obtain a subset of approximate isometric maps.

\begin{lemma}\label{l.countable}
Let $S$ be path-connected and $S_0\subset S$ be a countable dense subset for the induced metric $d_S$. Then,
\[
\bigcap_{x\neq y\in S_0} F_{\eps}(x,y,S) \subset \mathcal{I}_{2\,\eps}(S).
\]
\end{lemma}

\begin{proof}
We may assume without loss of generality that $\eps<1/2$, otherwise the statement is trivial.
Let $f\in F_{\eps}(x_0,y_0,S)$ for all $x_0,y_0\in S_0$.
For $x,y\in S$, we choose $\eta>0$ and $x_0,y_0\in S_0$ such that
\[
\eta<\frac{\eps}{4}\,d_S(x,y),
\]
and
\[
d_S(x,x_0)+d_S(y,y_0)<\eta. 
\]
We can find two curves $\gamma_1\in \Gamma_S(x_0,x)$ and $\gamma_2\in\Gamma_S(y,y_0)$ such that
\[
 \ell_g(\gamma_1)+\ell_g(\gamma_2) \leq d_S(x,x_0)+d_S(y,y_0) + \eta.
\]
Observe that
\begin{align*}
(1-2\eps)d_S(x,y)&\leq (1-\eps)d_S(x_0,y_0)-\eps d_S(x,y)+(1-\eps)\eta\\
&\leq (1-\eps) d_S(x_0,y_0)-2\eta(1+\eps),
\end{align*}
since
$$
(3+\eps)\eta\leq \eps d_S(x,y).
$$
Then we consider the concatenation $\tilde\gamma := \gamma_2 \cdot \gamma \cdot \gamma_1$ (i.e., the curve obtained by joining, in the order, the curves $\gamma_1$, $\gamma$ and $\gamma_2$), and note that $\tilde\gamma\in \Gamma_S(x_0,y_0)$. Using that 
\begin{align*}
\ell_g(\tilde\gamma)\leq \ell_g(\gamma)+2\eta
\end{align*}
and that $f\in F_{\eps}(x_0,y_0,S)$, we obtain
\begin{align*}
\ell(f\circ\,\gamma)& \geq \ell(f\circ\,\tilde\gamma) - 2\,\eta\\
& > (1-\eps)\,d_S(x_0,y_0)-\eps(\ell_g(\gamma)+2\eta)-2\,\eta\\
&\geq  (1-2\eps)\,d_S(x,y)-\eps\ell_g(\gamma)\\
& > (1-2\,\eps)\,d_S(x,y)-2\eps\ell_g(\gamma).
\end{align*}
This shows that $f\in F_{2\eps}(x,y,S)$. Since this holds for every $x,y\in S$, we conclude $f\in \cI_{2\eps}(S)$.
\end{proof}

\subsection{Closedness}
The following lemma shows that the sets of approximate isometric maps are $G_\delta$ sets.

\begin{lemma}\label{l.closed}
Let $S\subset M$ be compact. Then, for every $x,y\in S$ and $\eps>0$, $F_\eps(x,y,S)$ is open in $\cX(M,\R^n)$.
\end{lemma}

\begin{proof}
We show that $\cX(M,\R^n)\setminus F_\eps(x,y,S)$ is
closed under the uniform convergence induced by $D$.
To this aim, assume that $f_k\in \cX(M,\R^n)\setminus F_\eps(x,y,S)$
converges to $f$ uniformly in $M$.
By assumption, there exist $\gamma_k\in\Gamma_S(x,y)$ with
\[
\ell(f_k\circ \gamma_k)+\eps\,\ell_g(\gamma_k)\leq (1-\eps)\,d_S(x,y).
\]
In particular, the lengths $\ell_g(\gamma_k)$ are uniformly bounded.
Therefore, since we are considering curves in the compact set $S$,
we may extract a subsequence such that $\gamma_{k_j}\to\gamma\in\Gamma_S(x,y)$
uniformly.
This implies that also $f_{k_j}\circ\gamma_{k_j}$ converges uniformly to
$f\circ\gamma$.
Now, since the length is lower semicontinuous under uniform convergence,
we deduce that
\[
\ell(f\circ\gamma)+\eps\,\ell_g(\gamma)\leq (1-\eps)\,d_S(x,y).
\]
This implies that $f\in \cX(M,\R^n)\setminus F_\eps(x,y,S)$, hence $\cX(M,\R^n)\setminus F_\eps(x,y,S)$ is closed.
\end{proof}

\subsection{Locality}
The notion of isometric map is local in the following sense.

\begin{lemma}\label{l.local}
Let $\{U_\alpha\}_{\alpha\in\cA}$ be an open covering of $M$ such that every $U_\alpha$ is path-connected. Let $f\in \cX(M,\R^n)$ be such that $f\vert_{U_\alpha} \in \cI(U_\alpha)$. Then $f\in \cI(M)$.
\end{lemma}

\begin{proof}
We need to prove that, for a given curve $\gamma:[0,1]\to M$,
\[
 \ell(f\circ \gamma)= \ell_g(\gamma).
\]
Since $\gamma([0,1])$ is compact, we begin fixing a finite covering of $\gamma([0,1])$ by sets $U_{\alpha_j}$, $j=1,\ldots,m$. Using the uniform continuity of $\gamma$, we infer the existence of $\eta>0$ such that
\[
\forall \; t\in [0,1]\quad \exists \; j \in \{1,\ldots,m\} \quad \text{such that} \quad \gamma([t,t+\eta])\subset U_{\alpha_j}.
\]
We then choose any partition $0=t_0 \leq \ldots \leq t_m=1$ such that $|t_i-t_{i+1}|\leq\eta$. By the choice of $\eta$, for every $i=1,\ldots m-1$ there exists $j(i)$ such that $\gamma([t_i,t_{i+1}]) \subset U_{\alpha_{j(i)}}$. Therefore, from $f\vert_{U_\alpha}\in \cI(U_\alpha)$ we deduce that
\[
\ell(f\circ \gamma\vert_{[t_i,t_{i+1}]})=\ell_g(\gamma\vert_{[t_i,t_{i+1}]})\quad\forall \; i=1,\ldots, m-1,
\]
and therefore
\[
\ell(f\circ\,\gamma)=\sum_{i=0}^{m-1} \ell(f\circ \, \gamma\vert_{[t_i,t_{i+1}]}) = \sum_{i=0}^{m-1} \ell_g(\gamma_{[t_i,t_{i+1}]}) =\ell_g(\gamma).
\qedhere
\]
\end{proof}

\section{Locally strictly short extensions}\label{s:LSSE}

As mentioned in the introduction, given a short map $f:K \subset \R^n\to \R^n$ on a compact set $K$, $f$ will have a unique $1$-Lipschitz extension $\bar{f}$ to a possibly larger set containing $K$, namely
\[
C(f,K) := \bigcup_{H\in \cS} \textup{conv}(H),
\]
where $\cS := \{ H\subset K : f\vert_H \;\textup{is an affine isometry}\}$. Here "$f\vert_H$ affine" is understood in the sense that 
$f\vert_H(x) = A\,x +b$ for some $A \in O(n)$ and $b \in \R^n$. Then $f$ extends uniquely as a $1$-Lipschitz map on $C(f,K)$ and 
in particular $K \subset C(f,K)$.
In the following lemmas we prove two simple properties of $C(f,K)$, namely its compactness and a hull-type property.

\begin{lemma}\label{l.C(f,K)}
For every $K\subset \R^n$ compact and $f:K \to \R^n$ short, $C(f,K)\subset \R^n$ is compact.
\end{lemma}

\begin{proof}
We notice first that $C(f,K)$ is a bounded set. Therefore, we need only to show that it is closed.
Assume that $z^l \in \textup{conv}(H_l) \to z$.
Using Carath\'eodory's Theorem, we may assume without loss of generality that $H_l = \{y^l_0, \ldots, y^l_{n}\}$ and 
\[
z^l= \sum_{i=0}^{n} \lambda_{i}^{l} y_i^l,\quad\text{with }\,\sum_{i=0}^n\lambda_i^l=1,\;\lambda_i^l\geq0.
\]
By compactness (up to extracting subsequences which are not relabelled) we may infer that there exist $y_i \in \R^n$ and $\lambda_i\in [0,1]$ for $i=0,\ldots, n$ such that
\[
 \lim_{l\to+\infty}y_i^l = y_i\quad\text{and}\quad \lim_{l\to+\infty} \lambda_i^l = \lambda_i.
\]
Then, $z\in \textup{conv}(H)$ for $H:=\{y_0,\ldots, y_n\}$. Moreover, $H\in \cS$ because
\[
 |g(y_i)-g(y_j)| = \lim_{l\to+\infty}|g(y_i^l)-g(y_j^l)|=\lim_{l\to+\infty}|y_i^l-y_j^l|=|y_i-y_j|\quad\forall \;i,j.
\]
This shows that $z\in C(f,K)$, i.e.~$C(f,K)$ is closed.
\end{proof}

\begin{lemma}\label{l.C}
Let $f:K\to \R^n$ be a short map, with $K\subset\R^n$ compact and let $\bar{f}:C(f,K)\to \R^n$ be the unique $1$-Lipschitz extension of $f$ to $C(f,K)$. Then,
\[
 C(\bar{f},C(f,K))=C(f,K).
\]
\end{lemma}

\begin{proof}
It is enough to show that, for every $x,y \in C(f,K)$ such that $|\bar{f}(x)-\bar{f}(y)|=|x-y|$, it holds
\[
[x,y]:=\big\{\lambda\,x+(1-\lambda)\,y \,:\, \lambda\in[0,1] \big\}\subset C(f,K).
\]
Without loss of generality, we may assume that
\begin{equation}\label{e.wlog}
y=\bar{f}(y)=0.
\end{equation}
Set $H:=\big\{x_0,\ldots, x_l\big\}\subset K$, $l\leq n$, such that $f\vert_H$ is an affine isometry and $x=\sum_i \alpha_i \,x_i$ for positive $\alpha_i$ with $\sum_i \alpha_i=1$.
Note that in general $l$ may be different from $n$, because we assumed that $\alpha_i>0$ for every $i$.
Since $f\vert_H$ is affine and \eqref{e.wlog} holds, we have
\begin{equation*}
\left\vert \sum_i \alpha_i\,x_i \right\vert =  \left\vert \sum_i \alpha_i\,f(x_i) \right\vert.
\end{equation*}
Squaring we get
\begin{equation}\label{e.square}
 \sum_i \alpha_i^2\,|x_i|^2 + \,\sum_{i\neq j}\alpha_i\,\alpha_j\, x_i \cdot x_j =
\sum_i \alpha_i^2\,|f(x_i)|^2 + \,\sum_{i\neq j}\alpha_i\,\alpha_j\, f(x_i) \cdot f(x_j) .
\end{equation}
From \eqref{e.wlog} and $\Lip(\bar{f})\leq 1$, it follows that $|f(z)|\leq |z|$ for every $z \in K$. Recalling that $|f(x_i)-f(x_j)| = |x_i-x_j|$ for $x_i,x_j\in H$, this implies
\begin{align}\label{e.product}
f(x_i) \cdot f(x_j) &= \frac{1}{2}\Bigl(|f(x_i)|^2 + |f(x_j)|^2 - |f(x_i)-f(x_j)|^2 \Bigr)\notag\\
& \leq\frac{1}{2}\Bigl(|x_i|^2 + |x_j|^2 - |x_i-x_j|^2 \Bigr)\notag\\
& = x_i \cdot x_j.
\end{align}
Using \eqref{e.square} and \eqref{e.product} together (recall that $|f(z)|\leq |z|$ for every $z \in K$), we deduce that $|f(x_i)|=|x_i|$ for every $x_i \in H$. In particular, $\{0\}\cup H \in \cS$ and by definition
\[
[0,x]\subset \textup{conv}\big(\{0\}\cup H\big) \subset C(f,K). \qedhere
\]
\end{proof}

We now turn to the proof of Proposition \ref{p.LSSE0}. We start with a definition.Ê

\begin{definition}[LSSE]\label{d.LSSE}
Let $K\subset\R^n$ be a compact set and $f:K\to\R^n$ a short map. We say that $f$ is \textit{locally strict short extendable}, or briefly $f$ is LSSE, if there exists $h \in \cX(\R^n,\R^n)$ such that $h\vert_K=f$ and $\Lip(h\vert_A) <1$ for every $A \subset\subset \R^n \setminus K$.
\end{definition}

Clearly, if $f:K\to\R^n$ is LSSE, then $C(f,K)=K$.
We show that this is also a sufficient condition for $f$ to be LSSE.

\begin{proposition}\label{p.lsse condition}
For a short function $f:K\to\R^n$ the following are equivalent
\begin{itemize}
 \item[(a)] $f$ is LSSE;
 \item[(b)] for every $x\notin K$ there exists $p_x\in\R^n$ such that
\begin{equation}\label{e.point ext}
|p_x-f(y)|<|x-y|\qquad\forall\;y\in K;
\end{equation}
 \item[(c)] for every $x\notin K$, there exist at least two different $1$-Lipschitz extensions $f_1$, $f_2$ of $f$ to $K\cup\{x\}$. 
\item[(d)]\begin{equation}\label{e.(P)}
 x,y\in K:\; |f(x)-f(y)|=|x-y|\quad\Rightarrow \quad [x,y]\subset K.
\end{equation}
\end{itemize}
In particular, $f$ is LSSE if and only if $C(f,K)=K$.
\end{proposition}

\begin{proof}
To prove the equivalence between (a) and (b), assume that $h$ is a locally strictly short extension of $f$. Then, it follows from Definition~\ref{d.LSSE} that $p_x:=h(x)$ fulfills \eqref{e.point ext}. Conversely, if \eqref{e.point ext} holds, for every $x\notin K$ there exists $\delta_x>0$ such that
\begin{equation}\label{e.loc constant}
|p_x-f(y)|<|z-y|\qquad\forall\;y\in K\quad\text{and}\quad\forall\;z\in B_{\delta_x}(x)\subset \R^n\setminus K.
\end{equation}
For every $x\notin K$, we define the functions $f_x$ by
\[
f_x(w):=
\begin{cases}
f(w) & \text{if } w\in K,\\
p_x & \text{if } w\in B_{\delta_x}(x),
\end{cases}
\]
and consider $F_x$ an arbitrary $1$-Lipschitz extensions to the whole $\R^n$ given by Kirszbraun's Theorem \cite[2.10.43]{FedBook}.
Since $\R^n\setminus K$ is locally compact, there exist countably many $x_i$ such that
\[
\R^n\setminus K=\bigcup_{i=1}^\infty B_{\delta_{x_i}}(x_i).
\]
Setting $h:=\sum_i 2^{-i} F_{x_i}$, it is immediate to verify from \eqref{e.loc constant} that $h$ is a locally strictly short extension of $g$.

To show the equivalence between (b) and (c), note that, if the maps $x\mapsto q$ and $x\mapsto q'$ are two different extensions to $K\cup\{x\}$, then $p_x := \tfrac{q+q'}{2}$ satisfies \eqref{e.point ext}.
Vice versa, if \eqref{e.point ext} holds, then the continuous function
\[
\Phi(y):=\frac{|p_x-f(y)|}{|x-y|}
\]
satisfies $\max_{K}\Phi=1-\eta$ for some $\eta>0$.
Then, for every $z\in B_\delta(p_x)$ with $\delta\leq\tfrac{\eta}{2}\,\textup{dist}(x,K)$, the extension of $f$ given by $x\mapsto z$ is a $1$-Lipschitz extension of $f$:
\[
\frac{|z-f(y)|}{|x-y|}\leq \frac{|p_x-f(y)|+\delta}{|x-y|}\leq 1-\eta+\frac{\delta}{|x-y|}<1 \quad \forall\; y\in K.
\]
Note that, we have actually proven that (b) fails in a point $x$ if and only if (c) fails in the same point $x$.

\medskip

So far we have proved the equivalence of (a), (b) and (c). 
Next, it is clear that (b) implies (d).

To show the converse, we argue by contradiction and assume that (d) holds but (c) not, i.e. there exists $x\notin K$ such that $f$ admits a unique extension $\bar{f}:K\cup\{x\}\to\R^n$. Let $\bar{f}(x)=p_x$ and set 
\[
H := \{y\in K : |f(y)-p_x|=|y-x| \}.
\]
Note that $H$ is compact and, by the failure of (b) in $x$, $H \neq \emptyset$. Two cases can occur:
\begin{itemize}
 \item[(i)] $p_x \notin \textup{conv}(f(H))$;
 \item[(ii)] $p_x \in \textup{conv}(f(H))$.
\end{itemize}
In case (i), since $\textup{conv}(f(H))$ is compact, there exists $\eps,\tau>0$, and $\nu \in \sS^{n-1}$ such that
\[
p_x\cdot \nu > 2\,\eps+ f(y)\cdot \nu \quad \forall \; y\in H_\tau \cap K,
\]
where $H_\tau$ denotes an open $\tau$-neighborhood of $H$.
Moreover, by compactness of $K\setminus H_\tau$, there exists a $\delta>0$ such that
\[
 |f(y)-p_x|+\delta\leq |x-y|\quad\forall\;y\in K\setminus H_\tau.
\]
An elementary computation shows that $x \mapsto q_x := p_x - \eta\,\nu$ is a new $1$-Lipschitz extension of $f$ to $K \cup \{x\}$ if $\eta$ is chosen accordingly. Indeed, we have
\begin{align*}
 |f(y)-p_x +\eta\,\nu|^2 & = |f(y)-p_x|^2 + \eta^2 + 2\,\eta\, (f(y)-p_x) \cdot \nu\\
& \leq |f(y)-p_x|^2 + \eta^2 - 4\,\eta\,\eps\\
&\leq |y-x|^2+ \eta^2 - 4\,\eta\,\eps
\qquad\qquad\qquad\qquad
\forall\; y\in H_\tau\cap K,
\end{align*}
and
\[
 |f(y)-p_x +\eta\,\nu| \leq |x-y| -\delta + \eta \qquad \forall\; y\notin H_\tau.
\]
Hence, it suffices to choose
\[
 \eta<\max\big\{\delta,4\,\eps\big\}.
\]
This contradicts the assumption that $\bar{f}$ is the only $1$-Lipschitz extension to $K\cup\{x\}$ and gives the desired conclusion in case (i).

In case (ii), let $l\in \N$ be the minimum integer with the following property: there exist $l$ points $\{y_1, \ldots, y_l\} =: H'\subset H$ such that $p_x\subset \textup{conv}(f(H'))$. We claim that
\begin{equation}\label{e.short}
|f(y_i)-f(y_j)|<|y_i-y_j| \quad \forall\; y_i,y_j\in H'.
\end{equation}
Indeed, assume this is not the case, e.g.~$|f(y_1)-f(y_2)| = |y_1-y_2|$.
Then, since $p_x = \sum_i \alpha_i\,f(y_i)$ for positive $\alpha_i$ with $\sum_i \alpha_i=1$ and $f\vert_{[y_1,y_2]}$ is affine, we can set
\[
 z := \frac{\alpha_1\,y_1+\alpha_2\,y_2}{\alpha_1+\alpha_2}.
\]
By \eqref{e.(P)}, $[y_1,y_2] \subset K$, thus implying in particular that $z \in K$. Moreover, by comparing the congruent triangles $\{y_1,y_2,x\}$ and $\{f(y_1),f(y_2),p_x\}$ we deduce that $z\in H$. Since it is moreover easy to see that $p_x \in \textup{conv}(f(\{z,y_3,\ldots, y_l\}))$, we obtain a contradiction with the assumption that $l$ was the least number satisfying the above property.

To conclude we note that \eqref{e.short} implies that there exists a strictly short extension of $f\vert_{H'}$ to $H'\cup \{x\}$, denoted by $F:H'\cup\{x\}\to \R^n$. Clearly, $F(x)\neq p_x$ by the definition of $H$. This leads to a contradiction and concludes the proof. Indeed, set $F(x)=:q_x$ and $\nu := \tfrac{q_x-p_x}{|q_x-p_x|}$. Since $p_x \in \textup{conv}(f(H'))$, there exists $y \in H'$ such that
\[
 p_x \cdot \nu \geq f(y) \cdot \nu,
\]
which in turns implies 
\[
 |f(y)-q_x|\geq|f(y)-p_x|=|y-x|,
\]
against $\Lip(F) < 1$.
\end{proof}

\section{Density}\label{s:density}

In this section we set, referring to the notation of Section \ref{s.approx}, 
$$
M=\R^n,
$$
and define, for every $x,y\in K$ and $\eps>0$, $E_\eps(x,y, K)$ to be the restriction of maps from $F_\eps(x,y,K)$ to $K$, i.e. 
\[
E_\eps(x,y, K) := \{ h \in \cX(K, \R^n) : \exists\; f\in F_\eps(x,y,K) \textrm{ s.t. } f\vert_K =h\}.
\]

Our aim is to prove the following density result.  

\begin{proposition}\label{p:density}
Let $K\subset \R^n$ be a compact set. Then, for every $x,y\in K$ and $\eps>0$, the set $E_\eps(x,y, K)$
is dense in  $\cX(K,\R^n)$.
\end{proposition}

\subsection{Single lamination}
In this section we show the basic lamination construction which will be used to increase distances in one direction.
We consider functions of the following form:
\[
w(x)=A \,x+\zeta\,h(x\cdot \xi),
\]
where $A\in \R^{n\times n}$, $\xi, \zeta\in\R^n$ and $h:\R\to\R$
is the $1$-periodic extension of the following piecewise
linear function with slopes $\lambda_1<0<\lambda_2$,
\begin{equation}\label{e:h}
h(t)=
\begin{cases}
\lambda_1\,t &\text{for }\;0\leq t\leq \frac{\lambda_2}{\lambda_2-\lambda_1},\\
\lambda_2\,(t-1) &\text{for }\;\frac{\lambda_2}{\lambda_2-\lambda_1}<t\leq 1.
\end{cases}
\end{equation}
Note that $w$ is Lipschitz and piecewise affine in parallel strips, with
\begin{equation}\label{e:grad lam}
\nabla w(x)=
\begin{cases}
A+\lambda_1\, \zeta\otimes \xi &\text{for }\;k< x\cdot \xi< k+\frac{\lambda_2}{\lambda_2-\lambda_1},\\
A+\lambda_2\, \zeta\otimes \xi &\text{for }\; k+\frac{\lambda_2}{\lambda_2-\lambda_1}< x\cdot \xi<k+1,
\end{cases}
\quad\text{for all }\;k\in\Z.
\end{equation}

In what follows, a simplex is defined to be the closed convex hull of $n+1$ affinely
independent points in $\R^n$,
$T:=\overline{co}\{x_0,\dots,x_n\}$, and its barycenter is the point $\bar{x}:=\frac{1}{n+1}\sum_{j=0}^nx_j$.

\begin{proposition}\label{p:lamination}
Let $T$ be a simplex and $u$ be a strictly short affine map on $T$,
with $\nabla u\equiv A$ and $A^TA\leq (1-\theta_0)\,I$ for some $0<\theta_0<1$.
Then, for every $0<\theta<\theta_0$ and $\eta>0$, there exists
$v\in\Lip(T,\R^n)$ such that:
\begin{itemize}
\item[(i)]  $v=u$ on $\de T$;
\item[(ii)] $\|u-v\|_{C^0(T)}\leq \eta$;
\item[(iii)] $\Lip(v)\leq 1-\frac{\theta}{4}$;
\item[(iv)] $(1-2\theta)\int_0^1 |\dot\gamma(t)\cdot e_1|dt \leq
\ell(v\circ \gamma)$
for every rectifiable $\gamma:[0,1]\to T_\eta$, where $T_\eta$ is the
$(1-\eta)$-rescaled simplex with the same barycenter as $T$.
\end{itemize}
\end{proposition}

For the proof of the proposition we need the following elementary linear algebra lemma.

\begin{lemma}\label{l:linalg}
Let $A\in \R^{n\times n}$ and $\theta>0$ be such that $A^TA\leq (1-\theta)\,I$.
Then, there exists $\xi\in \R^n$ such that
\begin{equation}\label{e:linalg}
(1-\theta)\,|\zeta\cdot e_1|^2\leq
\big(A^TA+\xi\otimes\xi\big)\zeta\cdot\zeta\leq
(1-\theta)\,|\zeta|^2\quad\forall\,\zeta\in\R^n.
\end{equation}
\end{lemma}

\begin{proof}
Let $B=(1-\theta)I-A^TA$, so that, by assumption, $B\geq 0$.
First consider the case $B_{11}>0$ and set
$\xi:=\frac{1}{\sqrt{B_{11}}}\,Be_1$.
We claim that
\begin{align}
(B-\xi\otimes\xi)e_1&=0,\label{e:linalg1}\\
(B-\xi\otimes\xi)w\cdot w&\geq 0\quad\forall\;w\in\R^n.\label{e:linalg2}
\end{align}
Indeed, \eqref{e:linalg1} follows directly from the definition of $\xi$. To see
\eqref{e:linalg2}, notice that $B\geq 0$ implies,
for any $t\in \R$ and any $w\in\R^n$,
\begin{equation}\label{e:linalg3}
B(w+te_1)\cdot(w+te_1)=t^2(Be_1\cdot e_1)+2\,t\,(Be_1\cdot w)+(Bw\cdot w)
\geq 0.
\end{equation}
The fact that the above quadratic expression in $t$ is nonnegative is equivalent
to
\[
(Bw\cdot w)(Be_1\cdot e_1)-(Be_1\cdot w)^2\geq 0.
\]
On the other hand, by direct calculation
\[
(B-\xi\otimes\xi)w\cdot w=B_{11}^{-1}\left((Bw\cdot w)(Be_1\cdot e_1)-(Be_1\cdot
w)^2\right),
\]
thus leading to \eqref{e:linalg2}.
Similarly, if $B_{11}=0$, we set $\xi=0$. Then, \eqref{e:linalg1} and
\eqref{e:linalg2} still hold: indeed, the latter is trivially true by the
assumption on $A$ and the former follows from \eqref{e:linalg3} being $w$ and
$t$ arbitrary.

To conclude the proof of the lemma, note that \eqref{e:linalg1} and
\eqref{e:linalg2}
are equivalent to
\begin{align*}
\big(A^TA+\xi\otimes\xi\big)\,e_1&=(1-\theta) \, e_1,\\
\big(A^TA+\xi\otimes\xi\big)\,w\cdot w&\leq (1-\theta) \,
|w|^2\quad\forall\;w\in\R^n.
\end{align*}
Therefore, for a general $\zeta=t\,e_1+w$ with $w\perp e_1$, \eqref{e:linalg}
follows:
\[
(1-\theta)\,t^2\leq \big(A^TA+\xi\otimes\xi\big)\zeta\cdot\zeta=
(1-\theta)\,t^2+\big(A^TA+\xi\otimes\xi\big)w\cdot w\leq
(1-\theta)\,(t^2+|w|^2).
\]
\end{proof}

\begin{proof}[Proof of Proposition \ref{p:lamination}]
We show that a suitable truncation of a single lamination satisfies the conclusion of the
proposition. Fix $0<\theta<\theta_0$ and $\eta>0$, and note that $A^TA\leq (1-\theta)\,I$.
We split into two cases, depending on whether $\det A=0$ or $\det A\neq 0$.

\medskip

{\bf The case $\det A\neq 0$.} 
Let $\xi$ be the vector given by Lemma~\ref{l:linalg} and
consider $\zeta\in\R^n$ and $\lambda_1<0<\lambda_2$ such that
\begin{gather*}
\zeta=A^{-T}\xi\quad\text{and}\quad 2\,\lambda_i+\lambda_i^2\,|\zeta|^2=1,\quad
\text{if}\quad \det A\neq0,
\end{gather*}
Choose a cut-off function $\psi:T\to[0,1]$, $\psi\in C_c^\infty(T)$, such that
$\psi\equiv1$ on $T_\eta$
and fix a periodic piecewise affine functions $h$ with slopes $\lambda_1$
and $\lambda_2$ as in \eqref{e:h}.
We claim that, for $\mu$ large enough, the map
\[
v(x)=u(x)+\frac{\zeta}{\mu}\,h(\mu\,x\cdot \xi)\,\psi(x)
\]
satisfies the conclusions of the lemma.

Clearly, (i) follows from $\psi\in C_c^\infty(T)$.
Moreover, since $\|u-v\|_{C^0}\leq
\frac{\|h\|_{C^0}|\zeta|}{\mu}$,
choosing $\mu>\frac{\|h\|_{C^0}|\zeta|}{\theta}$,
also (ii) follows.
Next, notice that, by the choice of $\zeta$, for almost every $x\in T$,
\begin{align*}
\nabla v(x)^T\,\nabla v(x)=&{}
A^T A+ \big(h'(\mu\,x\cdot \xi)\,\psi(x)\big)\, A^T \zeta\otimes
\xi+\big(h'(\mu\,x\cdot \xi)\,\psi(x)\big)\,\xi\otimes
A^T\zeta+\\
&+\big(h'(\mu\,x\cdot \xi)\,\psi(x)\big)^2\,|\zeta|^2\,\xi\otimes
\xi+E_\mu(x)\notag\\
=&{}
A^T A+\Big(2\,h'(\mu\,x\cdot \xi)\,\psi(x)
+\big(h'(\mu\,x\cdot \xi)\,\psi(x)\,|\zeta|\big)^2\Big)\,\xi\otimes
\xi+E_\mu(x),
\end{align*}
where $E_\mu(x)$ is an error satisfying $\|E_\mu\|_{C^0}\leq
\frac{C_0}{\mu}$, for some $C_0$ depending on $h,\psi,\zeta$.
Hence, since $h'=\lambda_i$ and $0\leq\psi\leq1$,
\[
2\,h'(\lambda\,x\cdot \xi)\,\psi(x)
+\big(h'(\lambda\,x\cdot \xi)\,\psi(x)\,|\zeta|\big)^2\leq1\quad
\text{for a.e.~}\;x\in T.
\]
Then, for $\frac{C_0}{\mu}< \theta/2$, (iii) follows from the convexity of
$T$, since
\begin{align*}
\Lip(v)^2&=\esssup_{x\in T}\sup_{|\eta|=1}|\nabla v(x)\,\eta|^2=
\esssup_{x\in T}\sup_{|\eta|=1}
\left(\nabla v(x)^T\nabla v(x)\,\eta\cdot\eta\right)\\
&\leq \sup_{|\eta|=1}\big((A^T A+\xi\otimes \xi)\eta\cdot\eta\big)+
\sup_{x\in T}|E_\mu(x)|\leq 1-\theta+\frac{C_0}{\mu}<1-\frac{\theta}{2}\\
&\leq \left(1-\frac{\theta}{4}\right)^2.
\end{align*}

To prove (iv), let $\gamma:[0,1]\to T_\theta$ be a rectifiable curve and let
\[
0=t_0<t_1<\dots<t_N=1 
\]
be any partition of the interval $[0,1]$.
By adding more points if necessary, since $v$ is a single lamination in
$T_\eta$, we may assume that the restriction of $v$
onto each interval $\left[\gamma(t_j),\gamma(t_{j+1})\right]$ is affine.
Moreover, by the explicit formula \eqref{e:grad lam},
\[
v(\gamma(t_{j+1}))-v(\gamma(t_j))=\big(A+\lambda_i\,\zeta\otimes\xi\big)\big(\gamma(t_{j+1})-\gamma(t_j)\big),
\]
where $\lambda_i$ is chosen depending on which strip the segment lies in and,
in case the segment lies on the boundary of a strip,
i.e.~$\xi\cdot\big(\gamma(t_{j+1})-\gamma(t_j)\big)=0$, any value can be taken.
Therefore, in both cases, using \eqref{e:linalg} and
\begin{align*}\label{e:metric}
(A+\lambda_i\, \zeta\otimes \xi)^T(A+\lambda_i \,\zeta\otimes \xi)
&=
A^T A+ \lambda_i\, A^T \zeta\otimes \xi+\lambda_i\,\xi\otimes
A^T\zeta+\lambda_i^2\,|\zeta|^2\,\xi\otimes \xi\notag\\
&=A^T A+\xi\otimes \xi,
\end{align*}
we have
\begin{align*}
\left|v(\gamma(t_{j+1}))-v(\gamma(t_j))\right|&\geq
\sqrt{1-\theta}\left|(\gamma(t_{j+1})-\gamma(t_j))\cdot e_1\right|\\
&\geq (1-2\,\theta)\left|(\gamma(t_{j+1})-\gamma(t_j))\cdot e_1\right|.
\end{align*}
Summing and refining the partition ad infinitum, since the integral in (iv) is
the total variation of the curve $\gamma\cdot e_1$, we conclude the proof in
the case $\det A\neq 0$.

\medskip

{\bf The case $\det A=0$.}
In this case we consider
\[
\zeta\in\rm{Ker}(A^T)\quad\text{and}\quad-\lambda_1=\lambda_2=|\zeta|=1.
\]
Then, for $h$, $\psi$ and $v$ as above, we have for almost every $x\in T$,
\begin{align*}
\nabla v(x)^T\,\nabla v(x)=&{}
A^T A+ \big(h'(\mu\,x\cdot \xi)\,\psi(x)\big)\, A^T \zeta\otimes
\xi+\big(h'(\mu\,x\cdot \xi)\,\psi(x)\big)\,\xi\otimes
A^T\zeta+\\
&+\big(h'(\mu\,x\cdot \xi)\,\psi(x)\big)^2\,|\zeta|^2\,\xi\otimes
\xi+E_\mu(x)\notag\\
&{}=A^T A+\psi(x)^2\,\xi\otimes \xi+E_\mu(x),
\end{align*}
where $E_\mu(x)$ is again an error satisfying $\|E_\mu\|_{C^0}\leq
\frac{C_0}{\mu}$, for some $C_0$ depending on $h,\psi,\zeta$.
Since $0\leq \psi\leq 1$ and $\psi\equiv 1$ in $T_\eta$, the estimates
(i)-(iv) follows in the same way as before.
\end{proof}

\bigskip

\subsection{Triangulation and approximation of short maps}
In this subsection we construct a calibration in order to obtain sufficient control on curves in $\Gamma_K(x,y)$.
We start by proving an elementary result on piecewise affine approximations on triangulations.

Let $T=\overline{co}\{x_0,\ldots, x_n\}$ be a simplex and $\bar x$ its barycenter.
Given $u:T\to\R^n$, the affine interpolation of $u$ in $T$ is the function
\[
\bar u(x)=u(x_0)+A\,(x-x_0),
\]
where $A\in \R^n$ is such that $\bar u(x_i)=u(x_i)$ for every $i$ ($A$
always exists and is unique because the points $x_i$ are affinely independent).
Note that not every affine interpolation of a short map is short.
Consider, for example, the map $u:\R^2\to\R^2$, $u(x)=(|x|,0)$, and the simplex $T$ of vertices
$x_0=0$, $x_1=\frac{(1,\sqrt{3})}{2}$ and $x_2=\frac{(1,-\sqrt{3})}{2}$.
It turns out that the affine interpolation of $u$ in $T$ is given by
\[
\bar u(x)=
\left(
\begin{array}{cc}
2 & 0\\
0 & 0
\end{array}
\right)
x,
\]
so that $\Lip(\bar u)= 2$, although $u$ is short.

The following lemma provides a bound for the Lipschitz constant of $u-\bar u$.

\begin{lemma}\label{l:lip}
Let $T$ be a simplex and $r_1,r_2>0$
be such that $B_{r_1}(\bar x)\subset T\subset B_{r_2}(\bar x)$.
For every $u\in C^2(T,\R^n)$, the affine interpolation $\bar u$ in $T$ satisfies
\begin{equation}\label{e:lip}
\Lip(u-\bar u)\leq \frac{4\,r_2^2}{r_1}\, \|\nabla^2u\|_{C^0(T)},
\end{equation}
where
\[
\|\nabla^2 u\|_{C^0(T)}=\max_{x\in T}\sqrt{\sum_{i,j,l=1}^n\left(\frac{\de^2 u_l}{\de x_i\de x_j}(x)\right)^2}.
\]

\end{lemma}
\begin{proof}
Let $A=\nabla \bar u$. For every $B\in\R^{n\times n}$, denote by $L_B$ the
linear map given by $L_B(x)=B\,x$ and denote by $|B|=\max_{|\eta|=1}|B\,\eta|$
the operator norm.
We claim that
\begin{equation} \label{e:lip2}
|B-A|\leq \frac{2\,r_2}{r_1}\,\Lip(u-L_B).
\end{equation}
Indeed, let $\eta\in\sS^{n-1}$ be such that $|B-A|=(B-A)\,\eta$ and consider the
line $l_\eta=\{\bar x+t\,\eta:t\in\R\}$.
Clearly, by the convexity of $T$, $l_\eta$ intersects $\de T$ in two points,
\[
p=\sum_i \lambda_i \,x_i\quad\text{and}\quad
q=\sum_i \mu_i\,x_i,
\]
with $\lambda_i,\mu_i\geq0$ and $\sum_i\lambda_i=\sum_i\mu_i=1$.
Then, since $|p-q|\geq2\,r_1$, it follows that
\begin{align*}
|B-A|&=\frac{|(B-A)\,(p-q)|}{|p-q|}\leq\frac{|\sum_i(\lambda_i-\mu_i)
\,(B-A)(x_i-x_0)|}{2\,r_1}\\
&=\frac{|\sum_i(\lambda_i-\mu_i)\,\big\{(B\,x_i-u(x_i))-(B\,x_0-u(x_0)\big\}|}{2\,r_1}\\
&\leq \frac{\Lip(u-L_B)\sum_i(\lambda_i+\mu_i)|x_i-x_0|}{2\,r_1}
\leq\frac{2\,\Lip(u-L_B)\,\diam(T)}{2\,r_1}\\
&\leq \frac{2\,r_2}{r_1}\,\Lip(u-L_B).
\end{align*}
By convexity, for every $f\in C^1(T,\R^n)$,
\begin{equation}\label{e:f lip}
\Lip(f)=\max_{x\in T}|\nabla f(x)|.
\end{equation}
Set $B=\nabla u(y)$ such that
\[
\Lip(u-\bar u)=\max_{x\in T}|\nabla(u-\bar u)|=|B-A|.
\]
From \eqref{e:f lip} and \eqref{e:lip2}, we deduce \eqref{e:lip}:
\begin{align*}
\Lip(u-\bar u)&=|B-A| \stackrel{\eqref{e:lip2}}{\leq}
\frac{2\,r_2}{r_1}\,\Lip(u-L_B)\stackrel{\eqref{e:f lip}}{\leq}
\frac{2\,r_2}{r_1}\,\max_{x\in T}|\nabla u(x)-\nabla u(y)|\\
&\leq\frac{2 \,r_2}{r_1}\,\|\nabla^2 u\|_{C^0(T)}\,\max_{x\in T}|x-y|\leq
\frac{4 \,r_2^2}{r_1}\,\|\nabla^2 u\|_{C^0(T)}.
\end{align*}
\end{proof}

\begin{remark}
Actually, increasing the angle in $x_0$ in the example given above shows that
estimate \eqref{e:lip2} is optimal up to a multiplicative constant.
\end{remark}

In what follows, a triangulation $\mathcal{T}=\{T_i\}_{i\in\N}$ of $\R^n$ is defined
as a family of simplices such that $\cup_i T_i=\R^{n}$ and, for every $i\neq j$,
$T_i\cap T_j$ is a common face when not empty.
We call a triangulation periodic if there exist finitely many simplices
$T_1,\ldots,T_N$ such that $\mathcal{T}=\cup_{i=1}^N\{T_i+v: v\in \Z^{n}\}$.

Given a simplex $T=\overline{co}\{x_0,\ldots,x_n\}$, we consider the
$(n-1)$-dimensional supporting linear subspaces of its faces defined as follows:
for $\alpha=(\alpha_1,\ldots,\alpha_n)$ with $0\leq \alpha_1<\ldots<\alpha_n\leq n$, the corresponding supporting hyperplane is
given by
\[
V^T_{\alpha}=\text{Span}{\{x_{\alpha_2}-x_{\alpha_1},\ldots, x_{\alpha_n}-x_{\alpha_1}\}}.
\]
We denote by $\mathcal{N}_{\mathcal{T}}$ the set of all unit normals
to the supporting hyperplanes of simplices $T$ in $\mathcal{T}$,
\[
\mathcal{N}_{\mathcal{T}}=\big\{\nu \in \sS^{n-1}: \nu\perp V^T_\alpha\;
\text{for some }\alpha \,\text{and}\; T\in\mathcal{T} \big\}.
\]

\begin{proposition}\label{p:calib}
For every $\theta>0$, there exist $0<\delta<1$, a periodic triangulation
$\mathcal{T}$ and a function $\ph\in C^{\infty}(\R^n)$ such that:
\begin{gather}
\ph(l,\bar x)=l\qquad\forall\;\bar
x\in\R^{n-1},\;\forall\;l\in\Z;\label{e:Atri3}\\
0\leq \frac{\de \ph}{\de
x_j}(x)\leq\theta\qquad\forall\;x\in\R^n,\quad j\neq 1;\label{e:Atri2}\\
0\leq\frac{\de \ph}{\de
x_1}(x)\leq\frac{1}{1-\theta}\quad\forall\;x\in\R^n
\quad\text{and}\quad
\frac{\de \ph}{\de x_1}(x)=0\quad\forall\;x\in
\mathcal{F}_{\delta},\label{e:Atri1}
\end{gather}
where $\mathcal{F}=\cup_{T\in\mathcal{T}}\de T$ is the union of the faces of the
simplices of
$\mathcal{T}$ and $\mathcal{F}_{\delta}$ denotes its open
$\delta$-neighborhood.
\end{proposition}

\begin{proof}

\noindent{\bf Step 1: The existence of a transversal triangulation. }
We start showing the existence of a periodic triangulation $\mathcal{S}$ such that
every orthogonal vector $\nu\in \mathcal{N}_{\mathcal{S}}$ satisfies $\nu\cdot
e_1\neq0$, i.e.~such that $e_1$ is transversal to any supporting hyperplane.

To this aim, consider $\{T_1,\ldots,T_M\}$, a triangulation of $[0,1/2]^n$ which
can be extended to the whole $\R^n$ by
periodicity (that such triangulation does exist is a simple exercise), and set
\[
\mathcal{R}=\bigcup_{j=1}^M\{T_j+v/2:v\in\Z^n\}=\{R_i\}_{i\in\N}.
\]
For $w\in\R^n$ with $w\cdot e_1=0$, let $f_w:\R^n\to\R^n$ be the piecewise affine map given by, for
every $x=(x_1,\bar x)\in\R\times \R^{n-1}$,
\[
f_w(x)= x+h(x_1)\,w,
\]
where $h:\R\to\R$ is the 1-periodic extension of
\[
h(t)=
\begin{cases}
t &\text{if }\;0\leq t\leq \frac{1}{2},\\
1-t &\text{if }\;\frac{1}{2}<t<1.
\end{cases}
\]
Note that, $f_{w}\vert_{R_i}$ is linear for every simplex $R_i$ of
$\mathcal{R}$, so that $S_i=f_{w}(R_i)$ are also simplices.
Moreover, since $w\cdot e_1=0$, $f_w:\R^n\to\R^n$ is a periodic homeomorphism: for every integer vector $v\in\Z^n$, $f_w(x+v)=f_w(x)+v$.
Hence, $f_w(R_i+v)=f_w(R_i)+v$ implies that $\mathcal{S}=\{S_i\}_{i\in \N}$ is a
periodic triangulation of $\R^n$ as
well.

We claim that there exists $w \perp e_1$ such that $\mathcal{S}$ is
transversal to $e_1$.
Indeed, for every simplex $R_i$, $\nabla(f_{w}\vert_{R_i})=L_{\pm w}$, where $L_{{\pm} w}v=v{\pm} (v\cdot e_1)\,w$ and the sign is chosen
depending on the sign of $h'(x_1)$ for $x\in R_i$.
By simple linear algebra, using $L_{w}^{-1}=L_{-w}$ as $w\perp e_1$,
we infer that
\begin{equation*}
\cN_{\mathcal{S}}=\left\{L_{\pm w}^T\nu':\nu'\in\cN_{\mathcal{R}}\right\}.
\end{equation*}
Hence, $\nu\in \cN_{\mathcal{S}}$ is orthogonal to $e_1$
if and only if there exists $\nu'\in\cN_{\mathcal{R}}$ such that
\begin{equation}\label{e:generic}
0=\la L_{\pm w}^T\nu', e_1 \ra=\la \nu', L_{\pm w}e_1 \ra=\la \nu', e_1 \ra\pm\la \nu', w \ra.
\end{equation}
Now notice that, for a fixed $\nu'$ either the solutions $w\perp e_1$ satisfying
\eqref{e:generic} are affine $(n-2)$-dimensional subspaces or,
in the case $\nu'=e_1$, there are no solution.
Hence, relying on the fact that $\cN_{\mathcal{R}}$ is finite,
$\mathcal{R}$ being periodic,
one infers that for $\cH^{n-1}$-a.e.~$w\perp e_1$ no $\nu\in\cN_{\mathcal{S}}$
is orthogonal to $e_1$.

\bigskip

\noindent{\bf Step 2: Construction of a calibration.} From now on we fix a periodic transversal triangulation $\mathcal{S}=\{S_i\}_{i\in\N}$.
For every $\gamma>0$, we denote by $\cF_\gamma$ the open $\gamma$-neighborhood
of union of all faces of $\mathcal{S}$.
Consider the $C^\infty$ function $g:\R^n\to[0,1]$,
\[
g=\rho_{\gamma/2}*\chi_{(\R^n\setminus \cF_{3\gamma/2})},
\]
where $\rho\in C^\infty_c(B_1)$ is such that $\rho\geq0$, $\int\rho=1$ and,
as usual $\rho_r=r^{-n}\rho(\frac{x}{r})$.
Note that, since $\mathcal{F}$ is periodic, also $g$ is periodic and
\[
g\equiv 0\quad\text{on}\quad \cF_\gamma
\quad\text{and}\quad
g\equiv1\quad\text{on}\quad \R^n\setminus\cF_{2\gamma}.
\]
Set, for $x=(x_1,\bar x)\in\R\times \R^{n-1}$,
\[
f(x_1,\bar x):=\int_0^{x_1}g(t, \bar x) \,dt.
\]
Clearly $f$ is smooth and, by the periodicity of $g$, for $t\in[0,1)$ and
$l\in\Z$, (below the computation for $l\in\N$,
the other case being analogous), we have
\begin{align}\label{e:f-per}
f(t+l,\bar x)&=\int_0^{t+l}g(s,\bar x)\,ds=
\sum_{i=0}^{l-1}\int_{i}^{i+1}g(s,\bar x)\,ds+
\int_l^{t+l}g(s,\bar x)\,ds\notag\\
&=l\,f(1,\bar x)+f(t,\bar x).
\end{align}
For every $\bar x\in \R^{n-1}$,
setting $l_{\bar x}=\{(t,\bar x): 0\leq t\leq 1\}$, it holds
\begin{equation}\label{e:bound f}
f(1,\bar x)\geq 1- \cH^1(l_{\bar x}\cap \cF_{2\gamma}).
\end{equation}
Since no $\nu\in \cN_S$ is orthogonal to $e_1$,
each $l_{\bar x}$ intersects transversally
a bounded number of faces, so that there exists a constant $C>0$
such that $\cH^1(l_{\bar x}\cap \cF_{2\gamma})\leq C\gamma$ for every
$\gamma>0$.
By \eqref{e:bound f}, for $\gamma$ small enough, the function
\[
\psi(x_1,\bar x):=\frac{f(x_1,\bar x)}{f(1,\bar x)}
\]
is well defined and smooth.
From \eqref{e:f-per} it follows that
\[
\psi(l+t,\bar x)=l+\psi(t,\bar x).
\]
In particular, $\psi(l,\bar x)=l$ and $\nabla\psi$ is $\Z^n$-periodic with
$\frac{\de\psi}{\de x_1}(x)=f(1,\bar x)^{-1} g(x)$.
Therefore, from the choice of $g$, we have
\begin{equation}\label{e:psi}
0\leq \frac{\de\psi}{\de x_1}\leq\frac{1}{1-C\,\gamma}\quad\forall \; x\in\R^n
\quad\text{and}\quad
\frac{\de\psi}{\de x_1}(x)=0\quad\text{for}\;x\in\mathcal{F}_\gamma.
\end{equation}
Now, for every $k\in\N$, consider the horizontal rescaling $\tau_k:\R^n\to\R^n$ given by
$\tau_k(x_1,\bar x)=\left(\frac{x_1}{k},\bar x\right)$.
We claim that, for sufficiently large $k$,
\[
\ph(x):=k^{-1}\psi(k\,x_1,\bar x)
\quad\text{and}\quad \mathcal{T}=\{T_i\},
\;\text{with }\;T_i:=\tau_k(S_i),
\]
satisfy the conclusions of the proposition for a suitable $\delta$.
Indeed, $\mathcal{T}$ is clearly periodic and $\ph(l,\bar x)=k^{-1}\psi(k\,l,\bar x)=l$,
thus proving \eqref{e:Atri1}.
Setting $\cF'=\cup_i\de T_i$, from \eqref{e:psi} we deduce that
$\frac{\de\ph}{\de x_1}(x)=\frac{\de\psi}{\de x_1}(k\,x_1,\bar x)$ satisfy
\begin{equation}\label{e:c1}
0\leq \frac{\de\ph}{\de x_1}\leq\frac{1}{1-C\,\gamma}\quad\forall \; x\in\R^n
\quad\text{and}\quad
\frac{\de\ph}{\de x_1}(x)=0\quad\text{for}\;
x\in\cF'_{\gamma/k}=\tau_k(\cF_\gamma).
\end{equation}
Moreover, using the periodicity of $\nabla \psi$,
\begin{equation}\label{e:c2}
\left|\frac{\de\ph}{\de x_j}\right|\leq k^{-1}\|\nabla\psi\|_{C^0}, \quad
\forall\;j\neq 1.
\end{equation}
Given now $\theta>0$, we can choose $\gamma$, $k$ and $\delta$ in the following way:
\[
\gamma\leq \frac{\theta}{C}, \quad k\geq \frac{\|\nabla\psi\|_{C^0}}{\theta}
\quad\text{and}\quad \delta\leq {\frac{\gamma}{k}},
\]
so that, from \eqref{e:c1} and \eqref{e:c2}, the lemma follows.
\end{proof}

\begin{remark}\label{r:rescaling}
We note here that, given $\mathcal{T}$ and $\ph$ as in Proposition~\ref{p:calib},
for every $k\in \N$, the following functions and triangulations,
\[
\ph_k(x):=k^{-1}\ph(kx)\quad\text{and}\quad
\mathcal{T}^k=\{T^k_i\}_{i\in\N},
\]
where $T^k_i=k^{-1}T_i$, satisfy the same conclusions as in Proposition~\ref{p:calib}
with $\delta_k=\delta/k$:
\begin{gather}
\ph_k(l,\bar x)=l\quad\forall\;\bar x\in\R^{n-1},\;
\forall\;l\in\Z;\label{e:tri3}\\
\left|\frac{\de \ph_k}{\de x_j}\right|\leq\theta\quad\forall\;x\in\R^n,\;j\neq
1;\label{e:tri2}\\
0\leq \frac{\de \ph_k}{\de x_1}(x)\leq\frac{1}{1-\theta}\quad\forall\;x\in\R^n
\quad\text{and}\quad
\frac{\de \ph_k}{\de x_1}(x)=0\quad\forall\;x\in
\mathcal{F}_{\delta_k}.\label{e:tri1}
\end{gather}
\end{remark}

\bigskip

\subsection{Proof of Proposition~\ref{p:density}}

In light of the Kirszbraun extension theorem, it suffices to show that, given a short map $f:\R^n\to\R^n$ and $\eta>0$,
there exists $h\in F_\eps(x,y,K)$ such that $\|f-h\|_{C^0(K)}\leq \eta$.

\medskip

There is no loss of generality in assuming that $x=0$, $y=e_1$
and $K\subset B_R$ for some $R>0$.
We construct $h$ as the result of successive approximations.

\noindent{\bf Step 1: Mollification.} We consider first the map $f_1=(1-2\,\theta)\,\rho_{\theta}* f$, where
$\theta>0$ is a real number to be fixed later.
Clearly,
\begin{equation*}
f_1\in C^\infty(\R^n),\quad \Lip(f_1)\leq 1-2\,\theta,
\end{equation*}
and
\begin{align}\label{e:app1-2}
\| f-f_1\|_{C^0(B_{2R})}&\leq
\| f-\rho_{\theta}* f\|_{C^0(\R^n)}+
2\,\theta\,\|\rho_{\theta}* f\|_{C^0(B_{2R})}
\notag\\
&\leq
\theta\,\left(1+2\,\| f\|_{C^0(B_{2R+\theta})}\right).
\end{align}

\medskip

\noindent{\bf Step 2: Piecewise affine approximation.} 
Next, we approximate $f_1$ uniformly by a piecewise affine map $f_2$.
To this aim, consider the periodic triangulation $\mathcal{T}$ given by Proposition~\ref{p:calib}.
Note that, by periodicity, there exist $\sigma, r>0$ such that, for every $T_i\in\mathcal{T}$,
\begin{equation*}
B_r(\bar x_i)\subset T_i\subset B_{\sigma r}(\bar x_i),\quad\text{with $\bar x_i$ barycenter of }\; T_i.
\end{equation*}
Choose $k\in\N$ such that
\begin{equation}\label{e:k}
\frac{4\,r\,\sigma^2\,\|\nabla^2 f_1\|_{C^0(B_{2R})}}{k}<\theta
\quad\text{and}\quad
\frac{4\,r\,\sigma}{k}\leq \theta,
\end{equation}
and consider $f_2$ the piecewise affine approximation of $f_1$ subordinated to
the rescaled triangulation $\mathcal{T}^k$ in Remark~\ref{r:rescaling}.
From Lemma~\ref{l:lip}, it follows that 
$f_2\vert_{T^k_i}$ is short
for every $T^k_i\subset B_{2R}$ because
\begin{equation*}
\Lip(f_2\vert_{T^k_i})\leq \Lip(f_1)+\Lip((f_2-f_1)\vert_{T^k_i})
\leq 1-2\,\theta+\frac{4\,r\,\sigma^2\,\|\nabla^2 f_1\|_{C^0(B_{2R})}}{k}
\stackrel{\eqref{e:k}}{\leq} 1-\theta.
\end{equation*}
Moreover, always for $T^k_i\subset B_{2R}$,
\begin{equation}\label{e:app2-2}
\|f_2-f_1\|_{C^0(T^k_i)}\leq \big(\Lip(f_2)+\Lip (f_1)\big)\,\diam(T^k_i)\leq
\frac{4\,r\,\sigma}{k} \stackrel{\eqref{e:k}}{\leq} \theta.
\end{equation}

\medskip

\noindent{\bf Step 3: Laminations.} 
Finally, in every $T^k_i\subset B_{2R}$ we replace $f_2$ by the single lamination
construction in Proposition~\ref{p:lamination}.
Since the boundary data for each simplex is the same of $f_2$,
gluing all the constructions together, we obtain a short map $f_3$
defined on the union of the $T^k_i\subset B_{2R}$.
Moreover, we take $\theta$ small enough in order to assure that
the boundary of the rescaled simplices $T^k_{i\theta}$ by a factor $(1- \theta)$ and 
with the same barycenter of $T^k_i$ belongs to the
$\delta_k$-neighborhood of the faces of $\mathcal{T}^k$, i.e.~(notation as in Remark~\ref{r:rescaling})
\[
\de T^k_{i\theta}\subset \mathcal{F}_{\delta_k}.
\]
With this assumption, by Proposition~\ref{p:lamination},
the function $f_3$ satisfies:
\begin{equation}\label{e:app3}
\|f_3-f_2\|_{C^0(B_{R})}\leq\theta,\quad\Lip(f_3)\leq 1-\frac{\theta}{2}
\quad\text{and}\quad
(1-\theta)\int_0^1 |\dot\gamma_1(t)|dt \leq \ell(f_3\circ \gamma),
\end{equation}
for every rectifiable $\gamma=(\gamma_1,\ldots,\gamma_n):[0,1]\to \R^n\setminus
\mathcal{F}_{\delta_k}$.

We set $h:=f_3\vert_{K}$.
Clearly, from \eqref{e:app1-2}, \eqref{e:app2-2} and \eqref{e:app3},
it follows that 
\begin{equation}\label{e:uniform close}
\|\bar h-\bar f\|_{C^0(K)}\leq \theta (3+2\,\|\bar f\|_{C^0(B_{2R+\theta})}).
\end{equation}
So, up to choosing $\theta$ suitably small,
we need only to show that $h\in F_\eps(x,y,K)$.
Let $\gamma\in\Gamma_K(x,y)$.
We start noticing that there exist finitely many pairwise disjoint open intervals
$I_l, J_m\subset[0,1]$ such that
\[
\gamma(I_l)\subset \R^n\setminus \mathcal{F}_{\delta_k}
\quad\text{and}\quad
\gamma(J_m)\subset \mathcal{F}_{\delta_k},
\]
and
\begin{equation}\label{e:error}
\sum_l\ell(\gamma\vert_{I_l})+\sum_m\ell(\gamma\vert_{J_m})\geq \ell(\gamma)-\theta.
\end{equation}
Therefore, we can estimate the length of $h\circ \gamma$ as follows:
letting $\ph_k$ be the function in Remark~\ref{r:rescaling},
\begin{align}\label{e:l(h gamma)}
\ell\big(h\circ \gamma\big)&\geq\sum_{l}\ell\big(h\circ \gamma\vert_{I_l}\big)
\stackrel{\eqref{e:app3}}{\geq} (1-\theta)\sum_l\int_{I_l}|\gamma_1'|
\stackrel{\eqref{e:tri1}}{\geq} (1-\theta)^2\sum_l\int_{I_l}\left|\frac{\de\ph_k}{\de x_1}(\gamma)\,\gamma_1'\right|\notag\\
&\stackrel{\mathclap{\eqref{e:tri1}}}{=}
(1-\theta)^2\sum_l\int_{I_l}\left|\frac{\de\ph_k}{\de x_1}(\gamma)\,\gamma_1'\right|
+(1-\theta)^2\sum_m\int_{J_m}\left|\frac{\de\ph_k}{\de x_1}(\gamma)\,\gamma_1'\right|\notag\\
&\stackrel{\mathclap{{\eqref{e:tri1}+}\eqref{e:error}}}{\geq}\qquad
(1-\theta)^2\int_0^1\left|\frac{\de\ph_k}{\de x_1}(\gamma)\,\gamma_1'\right|-\theta\,(1-\theta)\notag\\
&\geq (1-\theta)^2{\int_0^1 \left[ |(\ph_k\circ \gamma)'| -\sum_{j=2}^n\left|\frac{\de\ph_k}{\de x_j}(\gamma)\,\gamma_j'\right|\right]}-\theta\,(1-\theta)\notag\\
&\stackrel{\mathclap{\eqref{e:tri2}}}{\geq} (1-\theta)^2\big(\ph_k(\gamma(1))-\ph_k(\gamma(0))\big)
-{(n-1)}\,\ell(\gamma)\,\theta\,(1-\theta)^2-\theta\,(1-\theta)\notag\\
&=(1-\theta)^2-\theta\,(1-\theta)\big[{(n-1)}\,\ell(\gamma)\,(1-\theta)+1\big].
\end{align}
Therefore, from \eqref{e:l(h gamma)} we deduce that
there exists $\theta=\theta(\eps)>0$ such that
$\ell(h\circ \gamma)\geq(1-\eps)$ if $\ell(\gamma)\leq \eps^{-1}$.
Since the condition defining $F_\eps(x,y,K)$ is always satisfied if
$\ell(\gamma)> \eps^{-1}$, this implies that $h\in F_\eps(x,y,K)$ and finishes
the proof.


\section{Typical extensions}\label{s.ext}
In this section we prove Theorem~\ref{t.extension} which we restate for convenience.

\begin{theorem}\label{t.extension2}
Let $f:K\to\R^n$ be a short map, with $K\subset\R^n$ compact.
Set
\[
X_{f} := \big\{F\in \cX(\R^n,\R^n) : F\vert_{C(f,K)} = \bar{f}\big\},
\]
where $\bar{f}$ denotes the unique short extension of $f$ to $C(f,K)$.
Then 
$$
X_f \cap \cI(\R^n\setminus C(f,K))\textrm{ is residual in $X_f$.}
$$
\end{theorem}

\begin{proof}
Let $\{B_i\}_{i\in\N}$ be a countable family of closed balls $B_i\subset\R^n\setminus C(f,K)$ whose interiors cover $\R^n \setminus C(f,K)$. By Lemmas~\ref{l.iso}, \ref{l.countable} and \ref{l.local}, we have that
\begin{equation*}
\mathcal{I}(\R^n\setminus C(f,K))\cap X_f \supset \bigcap_{k\in\N}\bigcap_{i\in\N} \bigcap_{x,y\in B_i\cap\Q^n}  F_{1/k}(x,y,B_i) \cap X_f.
\end{equation*}
Therefore, in view of Lemma~\ref{l.closed}, it is enough to prove that $X_f \cap F_{1/k}(x,y,B_i)$ is dense in $X_f$.
For simplicity of notation we drop the subscript $i$, $B_i=B$ and 
show that, for every $F\in X_f$, $\eta>0$ and $\eps>0$, there exists a map $G \in X_f \cap F_{\eps}(x,y,B)$ such that
\begin{equation}\label{e.uniform}
\|F-G\|_{C^0(\R^n)}\leq \eta.
\end{equation}
We divide the proof in several steps.

\medskip

\noindent{\bf Step 1: local strictly short approximation.}
By Lemma~\ref{l.C} and Proposition~\ref{p.lsse condition} we can fix a locally strictly short extension $h:\R^n\to\R^n$ of $\bar f$. Let $R>0$ be such that $C(f,K) \cup B \subset B_R$ and $\eta_1>0$ to be fixed later.

If $F\vert_{B_{2R}}\equiv 0$, set $F_1:=F$. Otherwise, assuming
that $F\vert_{B_{2R}}\not\equiv 0$, fix $t>0$ arbitrary such that
\[
t<\frac{\eta_1}{\|h\|_{C^0(B_{2R})}+\|F\|_{C^0(B_{2R})}},
\]
and define the function $F_1:B_{2R}\to \R^n$ given by $F_1=(1-t)F+t\,h$: clearly in either case
\begin{gather}
F_1\vert_{C(f,K)}=\bar{f}, \label{e:step1_0}\\
\|F-F_1\|_{C^0(B_{2R})}\leq \eta_1,\label{e:step1_1}\\
\Lip(F_1\vert_{B})\leq (1-t)\,\Lip(F)+t\,\Lip(h\vert_{B})\leq 1-\alpha,\label{e:step1_2}
\end{gather}
for some  $0<\alpha<1$, because $h$ is strictly short in $B$.

\medskip

\noindent{\bf Step 2: global extension.} Next we extend $F_1$ to the entire $\R^n$ keeping close to $F$. To this aim, consider the function $F':\R^n \setminus B_{2R} \to \R^n$ given by
\[
F'(x) := F \left(x\left(1-\frac{\tau}{1+|x|}\right)\right),
\]
for some $\tau>0$ to be fixed momentarily. It is simple to verify that
\begin{equation}\label{e.uniform'}
\|F-F'\|_{C^0(\R^n\setminus B_{2R})}\leq \tau.
\end{equation}
Moreover, $F'$ is locally strictly short: indeed,
\begin{align*}
|F'(x) - F'(y)| & \leq  \left|x\left(1-\frac{\tau}{1+|x|}\right) - y\left(1-\frac{\tau}{1+|y|}\right)\right|\allowdisplaybreaks\\
& = \left|(x-y)\left(1-\frac{\tau\,\big(1+|x|\big)}{\big(1+|x|\big)\big(1+|y|\big)}\right) + \frac{\tau\,x\,\big(|x|-|y|\big)}{\big(1+|x|\big)\big(1+|y|\big)}\right|\allowdisplaybreaks\\
&\leq |x-y|\left(1-\frac{\tau\,\big(1+|x|\big)}{\big(1+|x|\big)\big(1+|y|\big)}\right) + |x-y|\,\frac{\tau\,|x|}{\big(1+|x|\big)\big(1+|y|\big)}\allowdisplaybreaks\\
& =  |x-y|\left(1-\frac{\tau}{\big(1+|x|\big)\big(1+|y|\big)}\right) < |x-y|.
\end{align*}

Next, consider the map given by
\begin{equation*}
F'':=
\begin{cases}
F_1 & \text{in }\, B_{2R-\frac{2 \tau R}{1+2R}},\\
F' &\text{in }\,\R^n\setminus B_{2R}.
\end{cases}
\end{equation*}
We claim that $F''$ is locally strictly short outside $C(f,K)$. Since $F_1$ and $F'$ are locally strictly short, it is enough to consider $z\in \de B_{2R-\frac{2\tau\,R}{1+2R}}$ and $w\in \de B_{2R}$ and estimate $|F''(z)-F''(w)|$.
To this aim, we set $\tilde w := \frac{w}{|w|} (2\,R-\frac{2\,\tau\,R}{1+2R})$ and note that there exists $\beta(\tau,R)>0$ such that
\begin{equation}\label{e.quot}
\frac{|z-\tilde w|}{|z-w|} \leq 1 - \beta \quad \forall\; z\in \de B_{2R-\frac{2\tau\,R}{1+2R}},\;\forall\; w\in \de B_{2R}.
\end{equation}
Indeed, for every fixed $w \in \de B_{2R}$, one can consider the function $\Phi(z) := \frac{|z-\tilde w|}{|z-w|}$ and notice that $\Phi$ is continuous on $\de B_{2R-\frac{2\tau\,R}{1+2R}}$ and $\Phi(z) <1$ for every $z$. Therefore, by compactness of the sphere, $\Phi$ has a maximum which is strictly less then $1$ and is independent of $w$ because of rotational invariance.
We can, hence, estimate as follows:
\begin{align*}
|F''(z)-F''(w)|&\leq |F_1(z)-F_1(\tilde w)|+|F_1(\tilde w)-F'(w)|\allowdisplaybreaks\\
& = |F_1(z)-F_1(\tilde w)|+|F_1(\tilde w)-F(\tilde w)|\allowdisplaybreaks\\
&\stackrel{\eqref{e:step1_1}}{\leq} |z-\tilde w|+\eta_1 \allowdisplaybreaks\\
&\stackrel{\eqref{e.quot}}{\leq}
\left(1- \frac{\beta}{2}\right) |z-w|,
\end{align*}
provided $\eta_1 \leq \frac{\beta\,\tau\,R}{1+2R}$.
In particular, this implies that there exists $\theta>0$ such that
\[
\Lip\left(F''\vert_{{ (}B_{3R}\setminus B_{2R} { )}\cup B_{2R - \frac{2\tau R}{1+2R}}}\right) \leq 1 - \theta.
\]
Using the Kirszbraun extension theorem, we can hence extend $F''$ to a strictly short map $F'''$ on $B_{3R}$, and finally set
\[
F_2 :=
\begin{cases}
F''' & \text{in } \, B_{3R},\\
F' & \text{in } \, \R^n \setminus B_{3R}.
\end{cases}
\]
Observe that, by construction, 
\begin{equation}\label{e:F_2}
\Lip\left(F_2\vert_{B_{3R}}\right)\leq 1-\theta.
\end{equation}
Moreover, for every $z\in B_{2R}\setminus B_{2 R-\frac{2 \tau R}{1+2R}}$, setting $\tilde z := \frac{z}{|z|} (2\,R-\frac{2\,\tau\,R}{1+2R})$, we have
\[
|F_2(z)-F_1(z)|\leq |F_2(z)-F_2(\tilde z)|+|F_1(\tilde z)-F_1(z)|\leq 2\,|z-\tilde z|\leq \frac{4\,\tau\,R}{1+2\,R}< 2\, \tau.
\]
It follows, then, that
\begin{align}\label{e:step3 bis}
\|F_2-F\|_{C^{0}(\R^n)} &= \max\big\{2\,\tau +\|F_1-F\|_{C^{0}(B_{2R})},\|F'-F\|_{C^{0}(\R^n\setminus B_{2R})}\big\}\notag\\
& \leq 2\,\tau + \eta_1.
\end{align}

\medskip

\noindent{\bf Step 3: almost isometric approximation.}
Using Proposition~\ref{p:density}, we find $F^{iv}\in \cX(B_{2R},\R^n) \cap F_{\eps}(x,y,B)$ such that
\begin{equation}\label{e:step2}
\|F^{iv}-F_2\|_{C^0(B_{2R})}\leq \theta\,\eta_2,
\end{equation}
for some $\eta_2>0$ to be fixed soon.
For now we merely assume that  $\eta_2$ satisfies the following: setting $B= B_{r}(x)$, we require $B'=B_{r+\eta_2}(x)\subset {B_{2R}}\setminus C(f,K)$.

Next, we verify that the map
\begin{equation*}
F^{v} :=
\begin{cases}
F^{iv} & \text{in }\, B,\\
F_2 &\text{in }\,B_{2R}\setminus B',
\end{cases}
\end{equation*}
is Lipschitz continuous with $\Lip(F^v)\leq 1$.
Indeed, arguing as before, it is enough to consider the case of $z\in B$ and
$w\in B_{2R}\setminus  B'$ and estimate as follows:
\begin{align*}
|F^v(z)-F^v(w)|&\leq |F_2(z)-F_2(w)|+|F_2({z})-F^{iv}({z})|\\
& \stackrel{\eqref{e:F_2}+\eqref{e:step2}}{\leq} (1-\theta)\,|w-z|+\theta\,\eta_2\\
&\leq (1-\theta)\,|w-z|+\theta\,|w-z|\leq |w-z|.
\end{align*}
Using Kirszbraun's Theorem, we extend $F^v$ to a short map $F_3$ on the whole $\R^n$.
As before, for every $z\in B'\setminus B$, taking $w\in \partial B'$ with $|w-z|\leq \eta_2$, we get
\[
|F_3(z)-F_2(z)|\leq |F_3(z)-F_3(w)|+|F_2(w)-F_2(z)|\leq 2\,|w-z|\leq 2\,\eta_2.
\]
It follows, then, from \eqref{e:step2} that
\begin{equation}\label{e:step3}
\|F_2-F_3\|_{C^{0}(\R^n)} = \max\big\{\|F^{iv}-F_2\|_{C^{0}(B)}, 2\, \eta_2 \big\} \leq 2\,\eta_2.
\end{equation}

\medskip

We can now conclude that the function $G:=F_3$ is an approximation for our initial function $F$.
Indeed, $G\in X_f$ since by \eqref{e:step1_0} $G\vert_{C(f,K)}=F_1\vert_{C(f,K)}=\bar{f}$ and $\Lip(G) \leq 1$.
Moreover, $G\in F_{\eps}(x,y,B)$ because $G\vert_{B}=F_3\vert_{B}$ and $F_3 \in \Lip(B_{2R},\R^n) \cap F_{\eps}(x,y,B)$. Finally, putting together \eqref{e:step3 bis} and \eqref{e:step3}, we conclude \eqref{e.uniform} by choosing suitably $\tau, \eta_1$ and $\eta_2$ in this order.
\end{proof}

For later use we state the following immediate corollary of Theorem~\ref{t.extension}. 

\begin{corollary}\label{c:dir}
Let $\Omega\subset\R^n$ be an open and bounded set, and let $h:\Omega\to\R^n$ be a
given Lipschitz map with $\Lip(h)\leq L$ for some $L>0$.
Then, for every $\eta>0$ and $M>L$, there exists a map $g:\Omega\to\R^n$ such
that $g\vert_{\de\Omega}=h$, $\|g-h\|_{C^0(\Omega)}\leq \eta$ and
every rectifiable curve $\gamma:[0,1]\to\Omega$ satisfies
$\ell(g\circ \gamma)= M\,\ell(\gamma)$.
\end{corollary}

\begin{proof}
The proof follows easily applying Theorem~\ref{t.extension} to $K:=\de \Omega$ and  $f=g/M$ (note that from the condition $\Lip(h) \leq L <M$ it follows that $C(f, K)=K$).
\end{proof}


\section{Generic restrictions}\label{s:restriction}
In this section we prove Theorem~\ref{t.restriction}. We start with the following proposition on the genericity of LSSE maps.

\begin{proposition}\label{p.LSSE}
Let $K\subset R^n$ be a compact set. Then, the typical short map in $\cX(K,\R^n)$ admits an extension to the whole $\R^n$, which is locally strictly short on $\R^n \setminus K$.
\end{proposition}

\begin{proof}
We construct a residual set of LSSE maps in $\cX(K,\R^n)$. For every $\eps>0$, let
$K_\eps$ denote the open $\eps$-neighborhood of $K$.
Let moreover $\cG_{\eps} \subset \cX(K,\R^n)$ be the set of short maps $f:K\to \R^n$ with this property: there exists $L<1$ and there exists $h:\R^n \setminus K_\eps\to\R^n$ such that $\Lip(h) \leq L$ and 
\begin{equation}\label{e.distance}
|h(z)-f(y)|\leq L\,|z-y|\quad \forall\; z\in \R^n \setminus K_\eps, \,\forall \; y\in K.
\end{equation}
Note that $\cG_{\eps}$ is open in $\cX(K,\R^n)$: indeed, if $\|f' -f\|_{C^0(K)}\leq \tfrac{(1-L)\eps}{2}$, then, for $z\notin K_\eps$ and $y\in K$, we have
\begin{align*}
|h(z)-f'(y)| & \leq |h(z)-f(y)| + |f(y)-f'(y)|\\
& \leq L\,|z-y|+ \frac{1-L}{2}\,\eps \leq \frac{1+L}{2} |z-y|,
\end{align*}
thus implying that $f'\in \cG_\eps$ because $\frac{1+L}{2} < 1$. On the other hand, $\cG_\eps$ is also dense. Indeed, as a consequence of Kirszbraun's theorem all strictly short maps from $K$ to $\R^n$ belong to $\cG_\eps$, and the set of strictly short maps on a compact set is dense in the set of short maps (indeed, given $f \in \cX(K, \R^n)$, $\lambda\,f$ with $\lambda <1$ is strictly short and converges uniformly to $f$ as $\lambda$ tends to $1$).

We show that the residual set
\[
\cG := \bigcap_{\Q \ni \eps>0} \cG_\eps
\]
is made of LSSE maps, thus proving the proposition. Indeed, let $g \in \cG$. By definition, for every $\eps_k=2^{-k}$ there exists a function $h_k:\R^n \setminus K_{\eps_k}\to\R^n$ satisfying \eqref{e.distance}. Let $H_k$ be the Kirszbraun extension (i.e.~with optimal Lipschitz constant) of the map
\[
K \cup { \left(\R^n\setminus K_{\eps_k}\right)} \ni
x \mapsto 
 \begin{cases}
  h_k(x) & \text{if }\, x\in { \left(\R^n\setminus K_{\eps_k}\right)},\\
  g(x) & \text{if }\, x\in K.
 \end{cases}
\]
Note that by \eqref{e.distance} the maps $H_k$ are short.
Set
\[
f:=\sum_k 2^{-k}H_k.
\]
The function $f$ is a locally strictly short extension of $g$. Indeed, by construction $\Lip(f)\leq 1$ and $f\vert_K=g$. Moreover, for every open set $B$ with $B\cap K =\emptyset$, $\Lip(f\vert_B) <1$ because $\Lip(h_k)<1$ for every $k$ such that $B \subset { \left(\R^n\setminus K_{\eps_k}\right)}$.
\end{proof}

\begin{proof}[Proof of Theorem~\ref{t.restriction}]
Recall from Section \ref{s:density} that for every $x_i\neq x_j \in \Q^n$ and $\eps, R >0$ the set $E_\eps(x_i, x_j, \bar B_R)$ is defined as 
\[
E_\eps(x_i, x_j, \bar B_R) := \{ h \in \cX(K,\R^n) \; :\; \exists\; f \in F_\eps(x_i,x_j,\bar B_R)\textrm{ s.t. } f\vert_K = h \}.
\]
By Lemma~\ref{l.closed} and the openness of the restriction map (see \cite[Theorem~2.2]{Kopecka}), $E_\eps(x_i, x_j, \bar B_R)$ are open subsets of $\cX(K,\R^n)$.
Moreover, by Proposition~\ref{p:density}, these sets are also dense.

Let $\cL$ be the set of LSSE maps $g:K \to \R^n$ and recall that $\cL$ is residual in $\cX(K,\R^n)$ by Proposition~\ref{p.LSSE}. We claim that every map in the residual set
\begin{equation*}
\cF := \cL \cap \bigcap_{x_i\neq x_j\in \Q^n} \bigcap_{\Q\ni\eps>0} \bigcap_{R\in \N\setminus \{0\}} E_\eps(x_i, x_j, \bar B_R)
\end{equation*}
satisfies the conclusion of the theorem, i.e.~is the restriction of an isometric map of the entire space.

To show this, let $f \in \cF$. In view of Theorem~\ref{t.extension} and Proposition~\ref{p.lsse condition}, there exists an extension $F:\R^n\to\R^n$ of $f$ such that $F\vert_{\R^n\setminus K}\in \cI(\R^n\setminus K)$. We want to prove that actually $F\in \cI(\R^n)$.

Fix any curve $\gamma:[0,1]\to\R^n$. We can assume without loss of generality that $\gamma$ is parametrized by arc-length. 
Set 
$$U:=\gamma^{-1}(\R^n\setminus K)\textrm{ and }V:=\gamma^{-1}(K).
$$
Since $F\vert_{\R^n\setminus K}\in \cI(\R^n\setminus K)$, it follows that
$|(F\circ \gamma)'| =1$ for a.e.~$t\in U$.
We need only to show that $|(F\circ \gamma)'| =1$ for a.e.~$t\in V$.

We argue by contradiction. Assuming the above claim is false: there exists a compact set $W \subset V$ and $0<\eta<1$ such that
\[
\mathcal{L}^1(W)>2\,\eta \quad \text{and} \quad
|(F\circ \gamma)'| = |(f\circ \gamma)'| < 1-2\, \eta
\; \text{ for a.e. } t\in W.
\]
It then follows that
\begin{equation}\label{e:contradiction}
\int_0^1 |(F\circ \gamma)'(t)| \, dt
\leq 1 - 2\, \eta + (1-2\,\eta) \, 2\, \eta = 1 - 4\, \eta^2.
\end{equation}
Consider next a partition $t_0=0 <t_1 < \ldots < t_m =1$ such that
\begin{gather}
\gamma(t_i) \neq \gamma(t_{i-1}) \quad\forall \; i \in \{1, \ldots, m\}\label{e:partition1}\\
\sum_{i=1}^m |\gamma(t_i) - \gamma(t_{i-1})| > 1 - \eta^2.\label{e:partition3}
\end{gather}
Then, by elementary algebra, from \eqref{e:contradiction}, \eqref{e:partition1} and \eqref{e:partition3} it follows that
\begin{align*}
\min_{i \in \{1, \ldots, m\}} \frac{\ell(F\circ \gamma\vert_{[t_{i-1},t_{i}]})}{|\gamma(t_i) - \gamma(t_{i-1})|}
\leq \frac{\sum_{i=1}^m
\ell(F\circ \gamma\vert_{[t_{i-1},t_{i}]})}
{\sum_{i=1}^m|\gamma(t_i) - \gamma(t_{i-1})|}
\leq \frac{1-4\, \eta^2}{1 - \eta^2} < 1 - 3\, \eta^2.
\end{align*}
Let $j \in \{1, \ldots, m\}$ be such that
\begin{equation}\label{e:bo}
\frac{\ell(F\circ \gamma\vert_{[t_j , t_{j-1}]})}{|\gamma(t_j) - \gamma(t_{j-1})|} < 1 - 3\, \eta^2.
\end{equation}
Fix next $\eps>0$ satisfying the following conditions:
\begin{gather}
\eps\leq \eta^2 |\gamma(t_j) - \gamma(t_{j-1})| \label{e:eps1}\\
\eps\left(1+ \frac{\ell(\gamma) + \eps}{|x_j - x_{j-1}|} \right) \leq \eta^2\label{e:eps2}.
\end{gather}
Consider two points $x_j$ and $x_{j-1} \in \Q^n$ such that
\begin{equation}\label{e:x_j}
|\gamma(t_j) - x_j| + |\gamma(t_{j-1})- x_{j-1}| \leq \eps
\end{equation}
and, since $f \in \cF$, a function $\bar F \in F_\eps(x_j,x_{j-1},\bar B_R)$ such
that $\bar F\vert_K = f$.
Then, since $\bar F \vert_K = F \vert_K$ and $|(\bar F \circ \gamma)'\vert_U| \leq 1 = |(F \circ \gamma)'\vert_U|$, we deduce from \eqref{e:bo} that
\begin{equation}\label{e:contraction}
\frac{\ell(\bar F\circ \gamma\vert_{[t_{j-1},t_j]})}{|\gamma(t_j) - \gamma(t_{j-1})|} < 1 - 3\, \eta^2
\end{equation}
Let $\bar\gamma$ be the curve
obtained concatenating the straight segment from $x_{j-1}$ to $\gamma(t_{j-1})$,
$\gamma\vert_{[t_{j-1},t_j]}$ and the straight segment from $\gamma(t_j)$ to $x_j$, i.e.
\[
\bar \gamma := [\gamma(t_{j-1}), x_{j-1}] \cdot \gamma\vert_{[t_{j-1},t_j]} \cdot [x_j, \gamma(t_j)].
\]
Then we calculate:
\begin{align}\label{e:finale}
\frac{\ell(\bar F\circ \bar \gamma)}{|x_j - x_{j-1}|} &
\stackrel{\eqref{e:x_j}}{\leq} \frac{\ell(\bar F\circ {\gamma\vert_{[t_{j-1},t_j]}} + \eps)}{|\gamma(t_j) - \gamma(t_{j-1})| {-} \eps}\notag\\
& \leq \left(\frac{\ell(\bar F\circ \gamma\vert_{[t_{j-1},t_j]})}{|\gamma(t_j) - \gamma_(t_{j-1})|} + 
\frac{\eps}{|\gamma(t_j) - \gamma(t_{j-1})|}\right) \left( \frac{|\gamma(t_j) - \gamma(t_{j-1})|}{|\gamma(t_j) - \gamma(t_{j-1})| - \eps}\right)\notag\\
& \stackrel{\eqref{e:bo} + \eqref{e:eps1}}{\leq} 
\left(1-3\,\eta^2 + \eta^2\right) \frac{1}{1-\eta^2} < 1- \eta^2\notag\allowdisplaybreaks\\
&\stackrel{\eqref{e:eps2}}{\leq} 
1 - \eps\left(1+ \frac{\ell(\gamma\vert_{[t_{j-1},t_j]})+\eps}{|x_j-x_{j-1}|}\right)\notag\\
&\stackrel{\eqref{e:x_j}}{\leq} 
 1 - \eps - \frac{\eps\, \ell(\bar \gamma)}{|x_j-x_{j-1}|}.
\end{align}
On the other hand \eqref{e:finale} implies that $\bar F \notin F_\eps(x_j, x_{j-1}, \bar B_R)$, which is the desired contradiction.
\end{proof}


\section{Isometric embedding of Riemannian manifolds}
Now we proceed with the proof of Theorem~\ref{t.immersion}.
In this section $M$ is a smooth manifold of dimension $n$ (with or without boundary) and $g \in \mathcal{T}^2(M)$ is a continuous Riemannian metric (i.e.~a symmetric and positive definite $2$-tensor field).

\subsection{Locally strictly short maps}
The following general density result is used in the proof of Theorem~\ref{t.immersion}.
Denote by $\Lip_{<1,\textup{loc}}(M,\R^n)$ the space of \textit{locally strictly short maps}:
\[
\Lip_{<1,\textup{loc}}(M,\R^n)=\big\{f\in \cX(M,\R^n)\,:\,\Lip(f|_A)<1\quad\forall\,A\subset\subset
M\big\}.
\]

\begin{lemma}\label{l:density}
The set of locally strictly short maps $\Lip_{<1,\textup{loc}}(M,\R^n)$ is dense in $\cX(M,\R^n)$.
\end{lemma}

\begin{proof}
For every short map $f\in \cX(M,\R^n)$ and every $\eps>0$, we show that there
exists $h\in \Lip_{<1,\textup{loc}}(M,\R^n)$ such that $D(f,h)\leq \eps$.
Fix a point $p_0\in M$.
Without loss of generality, we may assume that $f(p_0)=0$.
We claim that the map
\[
h(p)=f(p)\left(1-\frac{\eps}{1+d_M(p,p_0)}\right)
\]
fulfills the requirements.
Observe first that
\[
D(h,f) = \sup_{p\in M}\frac{\eps\,|f(p)|}{1+d_M(p,p_0)}
= \sup_{p\in M^n}\frac{\eps\,|f(p)-f(p_0)|}{1+d_M(p,p_0)}
\leq\frac{\eps\,d_M(p,p_0)}{1+d_M(p,p_0)}\leq\eps.
\]
Therefore, we need only to show that $h\in \Lip_{<1,\textup{loc}}(M,\R^n)$.
To this end, setting for brevity of notation
$d(p)=d_M(p,p_0)$, we notice that for any $p,q\in M$,
\begin{align*}
h(p)-h(q)&=(f(p)-f(q))\left(1-\frac{\eps}{1+d(p)}\right)-
f(q)\left(\frac{\eps}{1+d(p)}-\frac{\eps}{1+d(q)}\right)\\
&=(f(p)-f(q))\left(1-\frac{\eps}{1+d(p)}\right)-
f(q)\,\frac{\eps\,(d(q)-d(p))}{(1+d(p))\,(1+d(q))}.
\end{align*}
Hence, it follows that
\begin{align}\label{e:est lip}
|h(p)-h(q)|&\leq
d_M(p,q)\,\left(1-\frac{\eps}{1+d(p)}+\frac{\eps\,d(q)}{(1+d(p))\,(1+d(q))}
\right)\notag\\
&=d_M(p,q)\,\left(1-\frac{\eps}{(1+d(p))\,(1+d(q))}
\right).
\end{align}
Given any compact set $A\subset\subset M$, there exists $C>0$ such that
$\sup_{p\in A} d_M(p,p_0)\leq C$.
It follows from \eqref{e:est lip} applied to $p,q\in A$ that
\[
\Lip(h\vert_A)\leq 1-\frac{\eps}{(1+C)^2}<1
\]
thus implying $h\in \Lip_{<1,\textup{loc}}(M,\R^n)$.
\end{proof}

\subsection{Local bi-Lipschitz approximations}
For the proof of Theorem~\ref{t.immersion} we need also the following simple technical lemma.

\begin{lemma}\label{l.local covering}
Let $(B, h)$ be a Riemannian manifold with continuous metric tensor $h$, where
$B \subset \R^n$ denotes either the ball $B_2$ centered at the origin or the half ball $B_2 \cap \{x_n \geq 0 \}$.
For every $\beta>0$ there exists $r \in (0,1)$ with this property:
for every $p \in \bar B_1 \cap B$ there exists a diffeomorphism $\Phi:B_r(p) \to U$ for some convex open set $U \subset \R^n$ such that $\Phi^* g_0 = h(p)$ with $g_0$ the standard flat Euclidean metric of $\R^n$ and $\Phi$ is bi-Lipschitz
with
\begin{equation}\label{e:bi-Lip}
\Lip(\Phi) \leq 1+ \beta \quad \text{and} \quad \Lip(\Phi^{-1}) \leq 1+ \beta.
\end{equation}
\end{lemma}

\begin{proof}
Let $G: B \to \R^{n \times n}_{\textup{sym},+}$ be the matrix-field corresponding to the metric tensor $h$: namely, $\R^{n \times n}_{\textup{sym},+}$ denotes the positive definite symmetric $n \times n$ matrices such that
\[
h(v,w) = (G\,v) \cdot w
\]
where we recall $\cdot$ is the standard scalar product in $\R^n$.
By the continuity of $G$ and the compactness of $\bar B_1 \cap B$, there exists $r \in (0,1)$ such that 
\begin{equation}\label{e:g cont}
\frac{G(y)}{(1+\beta)^2} \leq G(x) \leq (1+\beta)^2\, G(y) \quad \forall \; x,y \in \bar B_1 \cap B, \; d_h(x,y) \leq 4\,r
\end{equation}
where the above inequalities are meant in the sense of quadratic forms.

Fix now any $p \in \bar B_1 \cap B$. By the spectral theorem we can find $R \in O(n)$ and $D \in \R^{n \times n}$ a positive definite diagonal matrix such that $G(p) = R^T D^2 R$.
We can then define $\Phi$ to be the linear map
$\Phi(x) := L(x - p)$ where $L = R^T D^{-1}$.
Clearly $U := \Phi(B_r(p))$ is convex and it is very simple to verify that $\Phi^* g_0 = h(p)$: indeed for every $v, w \in \R^n$ 
\begin{align*}
h(p)\big(D\Phi(p) v,  D\Phi(p) w \big) & = G(p) R^T D^{-1} v \cdot  R^T D^{-1} w = v \cdot w.
\end{align*}
In order to estimate the Lipschitz constant of $\Phi$, consider two points $x, y \in B_r(p)$, $0<\eta <r$ arbitrary and $\gamma \in \Gamma(x,y,B)$ such that $\ell_h(\gamma) \leq  d_h(x,y) + \eta$.
Then for every $t \in [0,1]$ we have
\begin{align*}
d_h(p, \gamma(t)) & \leq d_h(p, \gamma(0)) + d_h(\gamma(0), \gamma(t))
\leq d_h(p, x) + \ell_h(\gamma)\\
& \leq d_h(p, x) + d_h(x, y) +\eta\leq 4\,r.
\end{align*}
Hence \eqref{e:g cont} is applicable and implies that $h(p) \leq (1+\beta)^2 h(\gamma(t))$ as quadratic forms, or equivalently $g_0 \leq (1+\beta)^2 (\Phi^{-1})^* h(\gamma(t))$.
One can therefore estimate
\begin{align}
|\Phi(x) - \Phi(y)| & \leq \ell_{g_0}(\Phi \circ \gamma) = \int_0^1 |(\Phi\circ \gamma)'(t)| \,dt\notag\\
& \leq (1+\beta)\int_0^1 |\gamma'(t)|_{h(\gamma(t))} \,dt = (1+\beta) \ell_h(\gamma)\notag\\
& \leq (1+\beta) \big(d_h(x,y) + \eta \big).
\end{align}
Since $\eta>0$ arbitrary, we conclude that $\Lip (\Phi) \leq 1+\beta$.
Vice versa we can consider two points $z, w \in U$ and the straight line $\sigma:[0,1] \to U$ connecting $z$ to $w$ (note $\sigma([0,1]) \subset U$).
Arguing as before, from \eqref{e:g cont} we have that $ h(\gamma(t)) \leq (1+\beta)(\Phi)^* g_0$ from which
\begin{align}
d_h(z,w) & \leq \ell_{h}(\Phi^{-1} \circ \sigma) = \int_0^1 |(\Phi^{-1}\circ \sigma)'(t)|_{h(\Phi^{-1}\circ \sigma(t))} \,dt\notag\\
& \leq (1+\beta)\int_0^1 |\sigma'(t)| \,dt = (1+\beta) |z-w|
\end{align}
i.e.~$\Lip(\Phi^{-1}) \leq 1+\beta$.
\end{proof}

\subsection{Proof of Theorem~\ref{t.immersion}}
We fix a smooth atlas $\{(A_i,\ph_i)\}_{i\in\N}$ of $M$
with the following properties:
\begin{itemize}
\item[(a)] $A_i\subset\subset M$;
\item[(b$_1$)] $\ph_i(A_i)= B_2\subset\R^n$ if $A_i\cap\de M=\emptyset$;
\item[(b$_2$)] $\ph_i(A_i)=B_2\cap\{x_n\geq0\}\subset\R^n$ if $A_i\cap\de M \neq \emptyset$;
\item[(c)] $\cup_{i\in\N}\ph_i^{-1}(B_1)=M$.
\end{itemize}
Set $C_i=\ph_i^{-1}(\bar B_1)$ and note that $C_i$ is compact in
$M$.
By Lemmas~\ref{l.iso}, \ref{l.countable},  \ref{l.closed} and \ref{l.local} we have that
\[
\cI_{M}=\bigcap_{i\in\N}\bigcap_{k\in\N}\bigcap_{x\neq y\in
D_i}F_{\frac{1}{k}} (x,y,C_i)
\]
where $D_i=\ph_i^{-1}(\Q^n\cap \bar B_1)$.
It is then enough to show that
$F_{\eps}(x,y,C_i)$ is dense in $\cX(M,\R^n)$ for every $\eps>0$ and every $x,y \in C_i$.
To simplify the notation, since from now on the subindex $i$ is fixed, we drop
it and, moreover, we write $B$ for either $B_2$ or $B_2\cap\{x_n\geq0\}$,
according to the case occurring in (b$_1$) or (b$_2$).

\medskip

We have then fixed the following notation:
\[
A \subset M, \quad \ph:A \to B \quad \text{and}\quad \ph^{-1}(\bar B_1 \cap B) = C.
\]
We have to show that, given $f\in \cX(M,\R^n)$ and $\eta>0$,
there exists $F\in F_{\eps}(x,y,C)$ such that $D(F,f)\leq\eta$.
We divide the proof in different steps.

\medskip

\noindent\textbf{Step 1: locally strictly short approximation.}
Recalling that by Lemma~\ref{l:density} the inclusion
\[
\Lip_{<1,\textup{loc}}(M,\R^n) \subset \cX(M, \R^n)
\]
is dense, we then find $f_0 \in \Lip_{<1,\textup{loc}}(M,\R^n)$ such that
$D(f_0,f) \leq \tfrac{\eta}{2}$.
By the definition of $\Lip_{<1,\textup{loc}}(M,\R^n)$, there exists $\alpha>0$ such that $\Lip(f_0\vert_{C})\leq 1-\alpha$. Clearly, there is no loss of generality in assuming that $\alpha<\eps$.

\smallskip

\noindent\textbf{Step 2: local bi-Lipschitz approximations.}
Let $\beta>0$ be a parameter to be fixed later and $h := (\ph^{-1})^* g$ the pull-back metric.
One can then apply Lemma~\ref{l.local covering} to $(B,h)$ and find $r>0$ which satisfies the conclusion therein.
By a simple volume argument (recall that $B$ is either $B_2$ or the half ball $B_2 \cap \{x_n \geq 0\}$) there exists a constant $N=N(n)$ depending
only on the dimension $n$, in particular not on $r$, such that we can cover
$B$ by $N$ families of pairwise disjoint open balls of radius $r$. More precisely, for $l =1, \ldots, N$ there exists $\cF_l=\{B_{r}(p_{l,i})\}_{i=1}^{m(l)}$ for some $m(l) \in \N$ and $p_{l,i} \in \bar B_1 \cap B$, such that
\[
B_{r}(p_{l,i})\cap B_{r}(p_{l,j})=\emptyset\quad \forall\; i\neq j
\quad\text{and}\quad B\subset\bigcup_{l=1}^N\bigcup_{i=1}^{m(l)} B_{l,i}.
\]
For every pair $(l, i)$ above we let $\Phi_{l,i}: B_r(p_{l,i}) \to U_{l,i} \subset \R^n$ be the bi-Lipschitz diffeomorphism given in Lemma~\ref{l.local covering}, and we set $A_{l,i} := \ph^{-1} (B_r(p_{l,i}))$.

\smallskip

\noindent\textbf{Step 3: iterative procedure.}
We construct the map $F:M \to \R^n$ as the result of an iterative procedure
which leads to a sequence of maps $f_0,f_1,\ldots,f_N:M\to\R^n$ (where $N$ is the number of the families of the covering in the previous step) such that $F = f_N \in F_\eps(x, y, C)$.

We set $\theta=\tfrac{\eta}{2N}$ and $f_0$ given in Step 1,
and construct the functions $f_1,\ldots,f_N$
recursively satisfing the following:
\begin{gather}
\Lip_g(f_k)\leq (1+\beta)^{3k}(1-\alpha)\label{e:step k_1}\\
D(f_k,f_0)\leq k\,\theta\label{e:step k_3/2}\\
\ell(f_k\circ \gamma)\geq
(1+\beta)\,(1-\alpha)\,\ell_g(\gamma)\label{e:step k_2}
\end{gather}
for every $k\geq1$ and every rectifiable curve $\gamma:[0,1]\to\cup_{l\leq
k}\cup_{i}A_{l,i}\subset C$.

Note that \eqref{e:step k_1} and \eqref{e:step k_3/2} are clearly satisfied by $f_0$.
Given $f_{k-1}$ satisfying \eqref{e:step k_1}, \eqref{e:step k_3/2} and \eqref{e:step k_2} (only if $k\geq1$),
we construct $f_{k}$ in the following way.
We consider the balls $B_r(p_{k,i})$ of Step 2 and set
$\psi_{k,i}:U_{k,i}\to\R^n$ given by
\[
\psi_{k,i}=f_{k-1}\circ \Phi_{k,i}^{-1}.
\]
Using the bound on the Lipschitz constant of $\Phi_{k,i}^{-1}$ in \eqref{e:bi-Lip} and \eqref{e:step k_1}, one can verify that
\begin{align*}
\Lip(\psi_{k,i})&\leq (1+\beta)^{3k-2}(1-\alpha).
\end{align*}
Hence we can use Corollary~\ref{c:dir} and construct a map $\chi_{k,i}:U_{k,i}\to\R^n$ such that
\begin{align}
\Lip(\chi_{k,i})&\leq
(1+\beta)\,\Lip(\psi_{k,i})=(1+\beta)^{3k-1}(1-\alpha),\label{e:chi lip}\\
\chi_{k,i}\vert_{\de U_{k,i}}&=\psi_{k,i}\vert_{\de U_{k,i}},\quad
\|\chi_{k,i}-\psi_{k,i}\|_{C^0(U_{k,i})}\leq\theta,\label{e:chi close}
\end{align}
and for every rectifiable curve $\tilde \gamma:[0,1]\to U_{k,i}$
\begin{equation}\label{e:chi iso}
\ell(\chi_{k,i}\circ\tilde\gamma)=(1+\beta)^{3k-1}(1-\alpha)\,\ell(\tilde\gamma).
\end{equation}
Then, we set $f_k:M\to\R^n$,
\[
f_k(x)=
\begin{cases}
f_{k-1}(x)&\text{if }\, x\in M\setminus \cup_{i=1}^{m(k)} A_{k,i},\\
\chi_{k,i}\circ\Phi_{l,i}(x) & \text{if }\, x\in A_{k,i}\;\text{for some }i=1, \ldots, m(k).
\end{cases}
\]
By \eqref{e:chi close} and the fact that the $\{A_{k,i}\}_{i}$ are disjoint open sets, $f_k$ is well-defined and satisfies \eqref{e:step
k_3/2} by triangular inequality.
Moreover \eqref{e:step k_1} follows from \eqref{e:bi-Lip} and \eqref{e:chi lip} straightforwardly.
For what concerns \eqref{e:step k_2} we argue as follows. Consider
$\gamma:[0,1]\to\cup_{l\leq k}\cup_iA_{l,i}$ rectifiable.
Set $I=\gamma^{-1}(\cup_i A_{k,i})$. Since the sets $A_{l,i}$ are open and disjoint,
$I$ is relatively open in $[0,1]$ and we can write $I=\cup_i J_i$ with $J_i$ disjoint relatively open sets such that $\gamma(J_i)\subset
A_{k,i}$ for every $i$.
Setting $\tilde \gamma_i=\Phi_{k,i}\circ \gamma\vert_{J_i}$, it follows
from the definition of $f_k$ that
\begin{align*}
\ell(f_k\circ\gamma\vert_{J_i}) & =\ell(\chi_i\circ \tilde\gamma_i)
\stackrel{\eqref{e:chi iso}}{=}
(1+\beta)^{3k-1}(1-\alpha)\,\ell(\tilde\gamma_i)\\
&\stackrel{\eqref{e:bi-Lip}}{\geq}
(1+\beta)^{3k-2}(1-\alpha)\,\ell_g(\gamma_i).
\end{align*}
On the other hand, let $H\subset [0,1]\setminus I$ denote the set of points $t$
such that $I$ has Lebesgue density $0$ at $t$ and there exist $(f_k\circ\gamma)'(t)$, $(f_{k-1}\circ\gamma)'(t)$ with
\[
|(f_{k-1}\circ\gamma)'(t)|\geq (1+\beta)(1-\alpha)|\gamma'(t)|_g.
\]
Note that $H$ has full measure in $[0,1]\setminus I$ thanks to the assumption
of \eqref{e:step k_2} for $f_{k-1}$.
Since $f_k\circ\gamma\vert_H=f_{k-1}\circ\gamma\vert_H$, it follows
easily that, for every $t\in H$,
\[
|(f_k\circ\gamma)'(t)|=|(f_{k-1}\circ\gamma)'(t)|\geq 
(1+\beta)\,(1-\alpha)\,|\gamma'(t)|_g.
\]
Therefore, \eqref{e:step k_2} for $f_k$ follows from
\begin{align*}
\ell(f_k\circ \gamma)&=\sum_i\ell(f_k\circ \gamma\vert_{J_i})+\int_{[0,1]\setminus
I}|(f_k\circ\gamma)'(t)|\,dt\\
&\geq
(1+\beta)^{3k-2}\,(1-\alpha)\sum_i\ell_g(\gamma_i)+
(1+\beta)\,(1-\alpha)\int_{H}|\gamma'(t)|_g\,dt\\
&\geq
(1+\beta)\,(1-\alpha)\left(\sum_i\ell_g(\gamma_i)+
\int_{H}|\gamma'(t)|_g\,dt\right)=(1+\beta)\,(1-\alpha)\,\ell_g(\gamma).
\end{align*}

Clearly $F=f_N$ concludes the proof for
\[
0< \beta<\sqrt[3N]{\frac{1}{1-\alpha}}-1.
\]
Indeed, Step 1, \eqref{e:step k_3/2} and \eqref{e:step k_1} imply
$D(F,f)\leq \eta$ and $\Lip(F)\leq1$. Moreover Step 2, $\alpha < \eps$ and \eqref{e:step k_2} lead easily to $F\in F_{\eps}(x,y,C)$. \qed

\bibliographystyle{plain}
\bibliography{reference-iso}

\begin{thebibliography}{10}

\bibitem{BalkaHarangi}
Rich{\'a}rd Balka and Viktor Harangi.
\newblock Intersection of continua and rectifiable curves.
\newblock {\em Proceedings of the Edinburgh Mathematical Society (Series 2)},
  57:339--345, 6 2014.

\bibitem{Brehm:1981dg}
Ulrich Brehm.
\newblock {Extensions of distance reducing mappings to piecewise congruent
  mappings on ${\bf R}^{m}$}.
\newblock {\em J. Geom.}, 16(2):187--193, 1981.

\bibitem{BressanFlores}
Alberto Bressan and Fabi{\'a}n Flores.
\newblock {On total differential inclusions}.
\newblock {\em Rend. Sem. Mat. Univ. Padova}, 92:9--16, 1994.

\bibitem{Cellina:2005tx}
Arrigo Cellina.
\newblock {A view on differential inclusions}.
\newblock {\em Rend. Semin. Mat. Univ. Politec. Torino}, 63(3), 2005.

\bibitem{Cellina:1995uv}
Arrigo Cellina and Stefania Perrotta.
\newblock {On a problem of potential wells}.
\newblock {\em J. Convex Analysis}, 1995.

\bibitem{DM97}
Bernard Dacorogna and Paolo Marcellini.
\newblock General existence theorems for {H}amilton-{J}acobi equations in the
  scalar and vectorial cases.
\newblock {\em Acta Math.}, 178(1):1--37, 1997.

\bibitem{DeBlasi:1982tk}
Francesco~Saverio De~Blasi and Giulio Pianigiani.
\newblock {A Baire category approach to the existence of solutions of
  multivalued differential equations in Banach spaces}.
\newblock {\em Funkcial Ekvac}, 1982.

\bibitem{FedBook}
Herbert Federer.
\newblock {\em Geometric measure theory}.
\newblock Die Grundlehren der mathematischen Wissenschaften, Band 153.
  Springer-Verlag New York Inc., New York, 1969.

\bibitem{GromovBook}
Mikhael Gromov.
\newblock {\em Partial differential relations}, volume~9 of {\em Ergebnisse der
  Mathematik und ihrer Grenzgebiete (3) [Results in Mathematics and Related
  Areas (3)]}.
\newblock Springer-Verlag, Berlin, 1986.

\bibitem{Kirchheim}
Bernd Kirchheim.
\newblock {Rigidity and Geometry of Microstructures}.
\newblock Habilitation Thesis, Univ. Leipzig, 2003.

\bibitem{Kopecka}
Eva Kopeck{\'a}.
\newblock Extending {L}ipschitz mappings continuously.
\newblock {\em J. Appl. Anal.}, 18(2):167--177, 2012.

\bibitem{Kuiper}
Nicolaas~H. Kuiper.
\newblock On {$C^1$}-isometric imbeddings. {I}, {II}.
\newblock {\em Nederl. Akad. Wetensch. Proc. Ser. A. {\bf 58} = Indag. Math.},
  17:545--556, 683--689, 1955.

\bibitem{MS96}
Stefan M{\"u}ller and Vladimir {\v{S}}ver{\'a}k.
\newblock Attainment results for the two-well problem by convex integration.
\newblock In {\em Geometric analysis and the calculus of variations}, pages
  239--251. Int. Press, Cambridge, MA, 1996.

\bibitem{MS99}
Stefan M{\"u}ller and Vladimir {\v{S}}ver{\'a}k.
\newblock Convex integration with constraints and applications to phase
  transitions and partial differential equations.
\newblock {\em J. Eur. Math. Soc. (JEMS)}, 1(4):393--422, 1999.

\bibitem{Nash}
John Nash.
\newblock {$C^1$} isometric imbeddings.
\newblock {\em Ann. of Math. (2)}, 60:383--396, 1954.

\bibitem{Petrunin}
Anton Petrunin.
\newblock Intrinsic isometries in {E}uclidean space.
\newblock {\em Algebra i Analiz}, 22(5):140--153, 2010.

\end{thebibliography}

\end{document}